\documentclass[12pt, reqno]{amsart}

\usepackage[english]{babel}
\usepackage{verbatim}
\usepackage[utf8]{inputenc}
\usepackage{amsmath}
\usepackage{graphicx}
\usepackage{amssymb}
\usepackage{amsrefs}
\usepackage[all]{xy}
\usepackage{hyperref}
\usepackage[margin=1in]{geometry}
\usepackage{url}
\usepackage{tikz}
\usetikzlibrary{arrows, backgrounds, fit, matrix, positioning}
\usepackage{mathtools}
\usepackage{enumerate}
\usepackage{multirow}

\usepackage{titletoc}
\contentsmargin{2.5em}

\newtheorem{thm}{Theorem}
\newtheorem{lem}[thm]{Lemma}
\newtheorem{cor}[thm]{Corollary}

\newtheorem{prop}[thm]{Proposition}

\addto\languageXYZ{}

\theoremstyle{remark}
\newtheorem{defn}[thm]{Definition}

\newtheorem{rmk}[thm]{Remark}

\newtheorem*{ex}{Example}
\newtheorem*{exs}{Examples}

\numberwithin{thm}{section} \numberwithin{equation}{section}

\newcommand{\mcE}{\mathcal{E}}

\newcommand{\mcT}{\mathcal{T}}
\newcommand{\mcK}{\mathcal{K}}
\newcommand{\mcC}{\mathcal{C}}
\newcommand{\mcU}{\mathcal{U}}
\newcommand{\CC}{\mathbb{C}}
\newcommand{\KK}{\mathbb{K}}
\newcommand{\HH}{\mathbb{H}}

\newcommand{\ZZ}{\mathbb{Z}}
\newcommand{\NN}{\mathbb{N}}

\DeclareMathOperator{\Cstar}{C^*}

\keywords{K-theory, kk-theory, smooth generalized crossed products, generalized Weyl algebras}
\subjclass[2010]{Primary 46L87; Secondary  19K35, 58B34}

\title{Bivariant K-theory of generalized Weyl algebras}

\author{Julio Guti\'errez}
\address{Instituto de Matem\'atica y Ciencias Afines (IMCA) Calle Los Bi\'ologos 245. Urb San C\'esar.
La Molina, Lima 12, Per\'u.}
\email{julio.gutierrez@imca.edu.pe}
\thanks{Julio Guti\'errez was supported by Cienciactiva CG 217-2014}

\author{Christian Valqui}
\address{Pontificia Universidad Cat\'olica del Per\'u, Secci\'on Matem\'aticas, PUCP,
Av. Universitaria 1801, San Miguel, Lima 32, Per\'u.}
\email{cvalqui@pucp.edu.pe}

\thanks{Christian Valqui was supported by PUCP-DGI-2017-1-0035}

\begin{document}

\begin{abstract} We compute the isomorphism class in $\mathfrak{KK}^{alg}$ of all noncommutative generalized Weyl algebras $A=\CC[h](\sigma, P)$,
where $\sigma(h)=qh+h_0$ is an automorphism of $\CC[h]$, except when $q\neq 1$ is a root of unity. In particular, we compute the isomorphism class in
 $\mathfrak{KK}^{alg}$ of the quantum Weyl algebra, the primitive factors $B_{\lambda}$ of $U(\mathfrak{sl}_2)$ and the quantum weighted projective
 lines $\mathcal{O}(\mathbb{WP}_q(k, l))$.
\end{abstract}

\maketitle

\setcounter{tocdepth}{2}
\tableofcontents


\section*{Introduction}

In \cite{MR2240217}, Cuntz defined a bivariant $K$-theory $kk^{alg}$ in the category $\mathfrak{lca}$ of locally convex algebras.
These are algebras $A$ that are complete locally convex vector spaces over $\CC$ with a jointly continuous multiplication. To a pair of locally convex
 algebras $A$ and $B$, there correspond abelian groups $kk_n^{alg}(A, B)$, $n\in\ZZ$ and there are bilinear maps
$$kk_n^{alg}(A, B)\times kk_m^{alg}(B, C)\to kk_{n+m}^{alg}(A, C)$$
for every $A, B$ and $C$ locally convex algebras and $m, n \in \ZZ$. Using this product, we can define a category $\mathfrak{KK}^{alg}$ whose objects are
locally convex algebras and whose morphisms are given by the graded groups $kk_{*}^{alg}(A, B)$. Then the bivariant $K$-theory $kk^{alg}$ can be seen as a
functor $kk^{alg}:\mathfrak{lca}\to \mathfrak{KK}^{alg}$. This functor is universal among split exact, diffotopy invariant and stable functors (see
Theorem 7.26 in~\cite{MR2340673}). In
particular, an isomorphism in $\mathfrak{KK}^{alg}$  induces an isomorphism in $\mathfrak{KK}^{\mathcal{L}_p}$ (see
Definition~\ref{def kk Lp}) and in $HP$.

Joachim Cuntz initiated the construction of different bivariant K-theories in several categories (see \cite{MR1456322}, \cite{MR2240217}
and~\cite{MR2207702}), and in~\cite{MR2240217} he proved that
the Weyl algebra $W=\CC\langle x, y|xy-yx=1\rangle$ is isomorphic to $\CC$ in $\mathfrak{KK}^{alg}$. By the universal property of $\mathfrak{KK}^{alg}$
this implies $\mathfrak{KK}^{\mathcal{L}_p}_0(\CC, W)=\ZZ$
and $\mathfrak{KK}^{\mathcal{L}_p}_1(\CC, W)=0$.

On the other hand, in \cite{MR3054304}, exact sequences analog to the Pimsner-Voiculescu exact sequence were constructed for smooth generalized crossed
products that satisfy the condition of being tame smooth. We shall consider generalized Weyl algebras over $\CC[h]$ which are smooth generalized crossed
products (but are not tame smooth in general).

\begin{defn} Let $D$ be a ring, $\sigma \in Aut(D)$ and
$a$ a central element of $D$. The generalized Weyl algebra $D(\sigma, a)$ is the algebra generated by $x$ and $y$ over $D$ satisfying
\begin{equation*}
xd=\sigma(d)x,\  yd=\sigma^{-1}(d)y,\ yx=a \text{ and } xy=\sigma(a)
\end{equation*}
for all $d\in D$.
\end{defn}

In this article, we compute the isomorphism class in $\mathfrak{KK}^{alg}$ of all non commutative generalized Weyl algebras $A=\CC[h](\sigma, P)$,
 where $\sigma(h)=qh+h_0$ is an automorphism of $\CC[h]$ and $P\in\CC[h]$, except when $q\neq 1$ is a root of unity. In the table below we list all posible
 cases for $A$ and our results.

\begin{table}[htb]
\begin{tabular}{| p{3.1cm} |  l |  l | p{1.8cm} | l |}
\hline
\multicolumn{2}{|c|}{Conditions} &  \multicolumn{2}{|c|}{Results}  &  Observation \\ \hline
\multirow{2}{*}{$P$ is constant} &  $P=0$  & $A\cong_{\mathfrak{KK}^{alg}}\CC$ & Prop  \ref{case1}& $A$ $\NN$-graded\\ \cline{2-5}
  & $P\neq 0$ & $A\cong_{\mathfrak{KK}^{alg}}S\CC \oplus \CC$ & Prop \ref{invariants-exceptional-case-1}& $A$ tame smooth\\ \hline
\multirow{4}{3.1cm}{$P$ is nonconstant with $r$ distinct roots}
 & $q$ not a root of unity & $A\cong_{\mathfrak{KK}^{alg}}\CC^{r}$ & Thm \ref{main_result} Prop \ref{case4} &\\ \cline{2-5}
 & $q=1$ and $h_0\neq 0$ & $A\cong_{\mathfrak{KK}^{alg}}\CC^{r}$ & Thm \ref{main_result} &\\ \cline{2-5}
 &  $q\neq 1$, a root of unity & \multicolumn{2}{|c|}{No result}  &\\ \cline{2-5}
 & $q=1$ and $h_0= 0$  & \multicolumn{2}{|c|}{No result} & $A$ commutative\\ \hline
\end{tabular}
\end{table}

A generalized Weyl algebra $A=\CC[h](\sigma, P)$ is tame smooth if and only if $P(h)\in \CC[h]$ is a non zero constant polynomial
(see Remark \ref{GWA_not_tame_smooth}).
Hence, if $P$ is a non constant polynomial, $A$ is a generalized crossed product that is not tame smooth, and so we cannot use the results of
\cite{MR3054304}. However, in most cases we can construct an explicit faithful representation of $A$, which allows us to follow the general
strategy of~\cite{MR2240217} and \cite{MR3054304}, in order to determine the $\mathfrak{KK}^{alg}$ class of $A$.

Our main result is Theorem \ref{main_result}, which computes the isomorphism class of $A$ in $\mathfrak{KK}^{alg}$ in the following two cases:
\begin{itemize}
\item $q=1$ and $h_0\neq 0$.
\item $q$ is not a root of unity and $P$ has a root different from $\frac{h_0}{1-q}$.
\end{itemize}
In each of these cases we construct an exact triangle
\begin{equation} \label{triangulo exacto}
S A \to   A_1 A_{-1}\overset0{\to} A_0\to A,
\end{equation}
 in the triangulated category $(\mathfrak{KK}^{alg}, S)$, where $A_n$ is the subspace of degree $n$ of the $\ZZ$-graded algebra $A$ (see
 Lemma~\ref{GWA-grading}).
 In order to construct the exact triangle in~\eqref{triangulo exacto} we follow the methods of  \cite{MR3054304}: we construct a linearly split extension
$$0\to \Lambda_A\to \mcT_A\to A \to 0$$
and prove
$$\mcT_A\cong_{\mathfrak{KK}^{alg}} A_0\quad \text{and} \quad \Lambda_A\cong_{\mathfrak{KK}^{alg}} A_1A_{-1}.$$

The exact triangle in~\eqref{triangulo exacto} yields $A\cong_{\mathfrak{KK}^{alg}} A_0\oplus S(A_1A_{-1})$. The main result now follows after we prove
 $A_1A_{-1}\cong_{\mathfrak{KK}^{alg}} S\CC^{r-1}$ in Proposition~\ref{invariants_A1A-1}, since $A_0=\CC[h]\cong_{\mathfrak{KK}^{alg}} \CC$.

The main result allows to compute the isomorphism class in $\mathfrak{KK}^{alg}$ of the quantum Weyl algebra, the primitive factors $B_{\lambda}$ of
$U(\mathfrak{sl}_2)$ and the quantum weighted projective lines $\mathcal{O}(\mathbb{WP}_q(k, l))$  (see \cite{MR2989456}).

For the sake of completeness we also discuss the case of $\NN$-graded and the case of tame smooth  generalized Weyl algebras.

In the case where $A=\bigoplus_{n\in\NN} A_n$ is an $\NN$-graded locally convex algebra it can be shown that $A\cong_{\mathfrak{KK}^{alg}} A_0$. This is the
case when
\begin{itemize}
\item $P$ is nonconstant, $q$ is not a root of unity and $P$ has only $\frac{h_0}{1-q}$ as a root or
\item  $P=0$.
\end{itemize}
In these cases we obtain $A\cong_{\mathfrak{KK}^{alg}}\CC$.

In the case where $P$ is a nonzero constant polynomial, $A$ is a tame smooth generalized crossed product and the results from \cite{MR3054304} apply. In this
case there is an exact triangle
\begin{equation}
S A\to A_0\overset0{\to} A_0\to A,
\end{equation}
in the triangulated category $(\mathfrak{KK}^{alg}, S)$ and we obtain $A\cong_{\mathfrak{KK}^{alg}}  S\CC\oplus\CC$.

In the case where $q=1$ and $h_0=0$, we have $\sigma=id$ and so $A\cong \CC[h, x, y]/(xy-P)$ is a commutative algebra. This case and the case where $q\neq 1$
is a root of unity remain open.

The article is organized as follows. In section~\ref{Basic results}, we recall basic results on locally convex algebras.
Lemma~\ref{algtensor} is a technical result which asserts that the projective tensor product of the Toeplitz algebra $\mcT$ with
an algebra with a countable basis is the algebraic tensor product. In section~\ref{definition of kk} we recall the definition and properties of $kk^{alg}$
following~\cite{MR2240217} and~\cite{MR2207702}.  In section~\ref{generalized Weyl algebra}, we define generalized Weyl algebras and construct explicit
faithful representations when $q=1$ and $h_0\neq 0$, and when $q$ is not a root of unity and $P$ has a root different from $\frac{h_0}{1-q}$. In
section~\ref{computations}, we compute the isomorphism class in $\mathfrak{KK}^{alg}$ of  all noncommutative generalized Weyl algebras
$A=\CC[h](\sigma, P)$  where
$\sigma(h)=qh+h_0$ except when $q\neq 1$ is a root of unity.


\section{Basic results on locally convex algebras}
\label{Basic results}
In this section we recall some constructions in the category of locally convex algebras that are needed for the definition of the
bivariant $K$-theory $kk^{alg}$.
We follow the discussions in \cite{MR3054304} and \cite{MR2240217}. In Lemma~\ref{algtensor} we prove that the projective tensor product of the Toeplitz algebra $\mcT$ with an algebra with a countable basis is the algebraic tensor product.


\subsection{Locally convex algebras}
\begin{defn}
A locally convex algebra $A$ is a complete locally convex vector space over $\CC$ which is an algebra such that for any continuous seminorm $p$ in $A$
there is a continuous seminorm $q$ in $A$ such that $p(ab)\leq q(a)q(b)$ for all $a, b\in A$. This is equivalent to requiring the multiplication to be
jointly continuous.

A seminorm $p$ of $A$ is called submultiplicative if $p(ab)\leq p(a)p(b)$, for all $a, b\in A$. If the topology of $A$ can be defined by a family of
submultiplicative seminorms we say that $A$ is an $m$-algebra.
\end{defn}

Morphisms in the category of locally convex algebras are continuous homomorphisms. We denote by $\otimes_{\pi}$ the projective tensor product of locally convex
vector spaces (see chapter 43 in \cite{MR0225131}). This is a completion of the algebraic tensor product. The projective tensor product of two locally convex
 algebras is again a locally convex algebra. The following are examples of locally convex algebras.
\begin{enumerate}
\item All algebras with a countable basis over $\CC$. These are locally convex algebras with the topology generated by all seminorms
(Proposition 2.1 in~\cite{MR2240217}). Examples include the Weyl algebra and generalized Weyl algebras over $\CC[h]$
(see Corollary~\ref{GWA_countable_basis}).

\item $\mcC^{\infty}([0, 1])$, which is a locally convex algebra with the family of seminorms
$$p_n(f)=||f||+||f'||+\frac{1}{2}||f''||+\dots +\frac{1}{n!}||f^{(n)}||$$
where $||f||=sup\{f(t)|t\in [0, 1]\}$.
\item We define $\CC[0, 1]$ as the (closed) subalgebra of $\mcC^{\infty}([0, 1])$ of functions with all derivatives vanishing at $0$ and $1$. This is
a nuclear topological vector space (see Definition 50.1 and Theorem 50.1 in \cite{MR0225131}) and therefore, for any locally convex algebra $A$ we have
 $\CC[0, 1]\otimes_{\pi} A=A[0, 1]$, where $A[0, 1]$ is the algebra of $\mcC^{\infty}$ functions with values in $A$ and all derivatives vanishing at $0$
 and $1$. We define $A[0, 1)$ and $A(0, 1)$ as the subalgebras of $A[0, 1]$ of functions that vanish at $1$, and at $0$ and $1$ respectively.

\end{enumerate}

\begin{defn}
We denote by  $SA$ and $CA$ the algebras $A(0, 1)$ and $A[0, 1)$ and we call them the suspension and the cone of $A$ respectively.
\end{defn}

Note that $S(\cdot)$ is a functor. Given a morphism of locally convex algebras $\phi:A\to B$, there is a morphism $S(\phi):SA\to SB$ defined by
$f\mapsto \phi\circ f$. We can iterate this functor $n$ times to obtain $S^nA$ and $S^{n}(f)$.


\subsection{Diffotopies}\

An important feature of
the bivariant $K$-theory $kk^{alg}$ is the invariance with respect to differentiable homotopies. For more details on diffotopies consult Section 6.1
in~\cite{MR2340673}.

\begin{defn}
Let $\phi_0, \phi_1: A\to B$ be homomorphisms of locally convex algebras. A diffotopy between $\phi_0$ and $\phi_1$ is a homomorphism
$\Phi:A\to \mcC^{\infty}([0, 1],B)$ such that $ev_i\circ \Phi=\phi_i$. If there is a diffotopy between $\phi_0$ and $\phi_1$ we call them
diffotopic and write $\phi_0\simeq \phi_1$.
\end{defn}

Using a reparameterization of the interval we can assume that all derivatives of $\Phi$ at $0$ and $1$ vanish and therefore we can assume that a diffotopy is
given by a map $\Phi:A\to B[0, 1]$. With this characterization we can show that diffotopy is an equivalence relation.

\begin{defn}
Given two locally convex algebras $A$ and $B$, we denote by $\langle A, B\rangle$ the set of diffotopy classes of continuous homomorphisms from $A$ to $B$.
Given $\phi:A\to B$ a continuous homomorphism, we denote by $\langle \phi \rangle$ its diffotopy class.
\end{defn}

\begin{lem}\label{group_structure}
There is a group structure in $\langle A, SB\rangle$ given by concatenation. The group structures in $\langle A, S^nB\rangle$ that we get from concatenation
in different variables all agree and are abelian for $n\geq 2$.
\end{lem}
\begin{proof}
See Lemma 6.4 in \cite{MR2340673}.
\end{proof}

\begin{defn}
A locally convex algebra $A$ is called contractible if the identity map is diffotopic to $0$.
\end{defn}

\begin{exs}
Examples of contractible locally convex algebras are $h\CC[h]$ and $CA$. The diffotopies are given by $\phi_s:h\CC[h]\to h\CC[h]$, $\phi_s(h)=sh$ and
$\psi_s:CA\to CA$, $\psi_s(f)(t)=f(st)$, respectively. Note that the algebras $(h-h_0)\CC[h]$ are isomorphic to $h\CC[h]$ and therefore are also contractible.
\end{exs}

We conclude this section with a note on $\NN$-graded algebras.

\begin{lem}\label{N-graded}
Let $A=\bigoplus_{n\in \NN} A_n$ be an $\NN$-graded locally convex algebra, then $A$ is diffotopy equivalent to $A_0$. In particular
$\CC[h]$ is diffotopy equivalent to $\CC$.
\end{lem}

\begin{proof}
The diffotopy is given by the family of morphisms $\phi_t:A\to A$, $t\in [0, 1]$, sending an element $a_n\in A_n$ to $t^{n}a_n$. When $t=1$ we recover the
identity and when $t=0$ the morphism is a retraction of $A$ onto $A_0$.
\end{proof}

\subsection{Extensions of locally convex algebras}\
\label{extlca}

In this section we define extensions of locally convex algebras and their classifying maps. Extensions play a key role in the definition of $kk^{alg}$.

\begin{defn}
An extension of locally convex algebras
$$0\to I\to E\to B\to 0$$
is linearly split if there is a continuous linear section $s:B\to E$. Similarly we say that an extension of length $n$
$$0\to I \to E_1 \to \cdots \to E_n \to B \to 0$$
is linearly split if there is a continuous linear map of degree $-1$ such that $ds+sd=id$, where $d$ is the differential of the chain complex.
\end{defn}

\begin{ex}
Let $A$ be a locally convex algebra. The extension
$$0\to SA\to CA\to A\to 0$$
is called the cone extension of $A$. It is linearly split with continuous linear section $s:A\to CA$ given by $a\in A\mapsto f\in CA$ with $f(t)=(1-\psi(t))a$,
where $\psi:[0, 1]\to [0, 1]$ is a $C^{\infty}$ bijection with $f(0)=0$, $f(1)=1$ and all derivatives vanishing at $0$ and at $1$.
\end{ex}

Now, we define the tensor algebra which has a universal property in the category of locally convex algebras. It is a completion of the usual algebraic tensor
algebra. Let $V$ be a complete locally convex vector space. The algebraic tensor algebra is defined as
$$T_{alg}V=\bigoplus_{n=1}^{\infty}V^{\otimes n}.$$
Notice that we are considering a non-unital algebraic tensor algebra. There is a linear map $\sigma:V\to T_{alg}V$ mapping $V$ into the first summand. We
topologize $T_{alg}V$ with all seminorms of the form $\alpha\circ \phi$,
where $\phi$ is any homomorphism from $T_{alg}V$ into a locally convex algebra $B$ such that $\phi\circ\sigma$ is continuous on $V$ and $\alpha$ is a
continuous seminorm on $B$.
\begin{defn}
The tensor algebra $TV$ is the completion of $T_{alg}V$ with respect to the family of seminorms $\{\alpha\circ \phi\}$ defined above.
\end{defn}

The tensor algebra $TV$ is a locally convex algebra that satisfies the following universal property.

\begin{prop}
Given a continuous linear map $s:V\to B$ from a complete locally convex vector space $V$ to a locally convex algebra $B$ there is a unique morphism of locally
convex algebras $\tau:TV\to B$ such that $\tau\circ\sigma=s$. The morphism $\tau$ is defined by
 $\tau(x_1\otimes x_2\otimes\dots\otimes x_n)=s(x_1)s(x_2)\dots s(x_n)$ where $x_i\in V$.
\end{prop}

\begin{proof}
See Lemma 6.9 in \cite{MR2340673}.
\end{proof}

In particular, if $A$ is a locally convex algebra, the identity map $id:A\to A$ induces a morphism $\pi: TA\to A$.

We use the universal property of $TA$ to construct a universal extension. There is an extension
$$
0 \to JA \to TA \overset{\pi}{\to}  A \to 0
$$
where $JA$ is defined as the kernel of $\pi: TA\to A$, which has a canonical continuous linear section given by $\sigma:A\to TA$. This extension is universal in the
sense that given any extension of locally convex algebras $0\to I\to E \to B \to 0$ with continuous linear section $s$ and a morphism $\alpha: A\to B$, there
is a morphism of extensions
$$
\xymatrix{
0\ar[r] & JA \ar[r] \ar[d]_{\gamma}  & TA \ar[r]\ar[d]_{\tau} & A \ar[r]\ar[d]_{\alpha}  & 0 \\
0\ar[r] & I\ar[r]  & E \ar[r] &  B\ar[r] & 0
}
$$
where $\tau:TA\to E$ is the morphism induced by the continuous linear map $s\circ \alpha: A\to E$ and $\gamma: JA\to I$ is the restriction of $\tau$.

Notice that $J(\cdot)$ is a functor. Given a morphism $\alpha:A\to B$, consider the extension $0\to JB\to TB\to B\to 0$ with its canonical continuous linear
 section. Then we define $J(\alpha):JA\to JB$ in the natural way. We can iterate this construction $n$ times to obtain $J^nA$ and $J^{n}(\alpha)$.

We observe that the map $\gamma:JA\to I$ is unique up to diffotopy. Given two continuous linear sections $s_1$ and $s_2$,  the smooth family of continuous
linear sections $s_t=ts_1+(1-t)s_2$  induces a diffotopy $\gamma_t$ which connects $s_1$ and $s_2$. Hence the corresponding $\gamma$'s are diffotopic.

\begin{defn}
The morphism $\gamma:JA\to I$ is called the classifying map of the extension $0\to I\to E\to B\to 0$ and the morphism $\alpha:A\to B$. It is well defined up
 to diffotopy.
\end{defn}

Similarly, we can define the classifying map of an extension
$$0\to I \to E_1 \to \cdots \to E_n \to B \to 0$$
and a morphism $\alpha:A\to B$ to be the map $\gamma:J^{n}A\to I$ in
$$
\xymatrix{
0\ar[r] & J^{n}A \ar[r]\ar[d]_{\gamma}  & T(J^{n-1}A) \ar[r] \ar[d] & \cdots \ar[r]&  TA \ar[r] \ar[d] & A \ar[r]\ar[d]_{\alpha}  & 0 \\
0\ar[r] & I \ar[r]  & E_1 \ar[r] & \cdots \ar[r] & E_n\ar[r] &  B\ar[r] & 0
}
$$
which is also unique up to diffotopy.


\subsection{The algebra of smooth compact operators and the smooth Toeplitz algebra}\

We define the algebra $\mcK$ of smooth compact operators which play the role of the $\Cstar$-algebra of compact operators used in Kasparov's $KK$-theory.
Then we define the smooth Toeplitz algebra $\mcT$ and prove that the projective tensor product of $\mcT$ with
an algebra with a countable basis is the algebraic tensor product.

\begin{defn}
The algebra of smooth compact operators $\mcK$ is defined as the algebra of $\NN\times \NN$ matrices $a=(a_{ij})$ such that
$p_n(a)=\sum_{i, j\in \NN} (1+i+j)^n|a_{i, j}|$ is finite for $n\in \NN$. The topology is defined by the seminorms $p_n$.
\end{defn}

The algebra $\mcK$ with the seminorms $p_n$ is a locally convex algebra, which is isomorphic to the space $s$ of rapidly decreasing sequences as a locally
convex vector space.

\begin{lem}\label{spaces_isomorphic_to_s}
The locally convex spaces $\mcK$, $s\otimes_{\pi} s$ and $s\oplus s$  are isomorphic to $s$.
\end{lem}

\begin{proof}
The proofs of these facts can be found in \cite{MR671092} Chapter 3 Section 1.1.
\end{proof}

We also define the smooth Toeplitz algebra which plays the role of the Toeplitz $\Cstar$-algebra. The Fourier series gives an isomorphism of locally convex
spaces between $C^{\infty}(S^1)$ and the space of rapidly decreasing Laurent series (see Theorem 51.3 in \cite{MR0225131})
$$C^{\infty}(S^{1})\cong\left\{\sum_{i\in \ZZ}a_iz^{i} \ | \ \sum_{i\in \ZZ} |1+i|^n|a_i|<\infty, \ \forall n\in \NN\right\},$$
where $z$ corresponds to the function $z:S^{1}\to \CC$, $z(t)=t$. This space is isomorphic to the space $s$ of rapidly decreasing sequences.

\begin{defn}
The smooth Toeplitz algebra $\mcT$ is defined by the direct sum of locally convex vector spaces $\mcT=\mcK\oplus C^{\infty}(S^{1})$. In order to define the
multiplication, we define $v_k=(0, z^{k})$ and write $x$ for an element $(x, 0)$ with $x\in \mcK$. We denote the elementary matrices in $\mcK$ by $e_{ij}$ and
set $e_{ij}=0$ for all $i, j<0$. The multiplication is defined by the following relations
$$e_{ij}e_{kl}=\delta_{jk}e_{il},\quad v_ke_{ij}=e_{(i+k), j},\quad e_{ij}v_k=e_{i, (j-k)},$$
for all $i, j, k, l\in \ZZ$ and
$$v_kv_{-l}=\begin{cases} v_{k-l}(1-e_{00}-e_{11}-\dots e_{l-1, l-1})&, l>0 \\ v_{k-l} &, l\leq 0,\end{cases}$$
for all $k, l\in \ZZ$.

Denote $v_1$ and $v_{-1}$ by $S$ and $S^*$ respectively.
\end{defn}

There is a linearly split extension
$$0\to \mcK\to \mcT\to C^{\infty}(S^{1}) \to 0$$
where the continuous linear section $C^{\infty}(S^{1})\to \mcT$ is defined by $z\mapsto S$.

The smooth Toeplitz algebra is generated, as a locally convex algebra, by $S$ and $S^*$. In fact, it satisfies a universal property in the category of
$m$-algebras.

\begin{lem}[Satz 6.1 in \cite{MR1456322}]\label{universal_smooth_toeplitz}
$\mcT$ is the universal unital $m$-algebra generated by two elements $S$ and $S^{*}$ satisfying the relation $S^{*}S=1$ whose topology is defined by a
family of submultiplicative seminorms $\{p_n\}_{n\in \NN}$ with the condition that there are positive constants $C_n$ such that
\begin{equation}\label{bounded_conditions_toeplitz}
p_n(S^k)\leq C_n(1+k^n) \quad \text{and} \quad p_n(S^{*n})\leq C_n(1+k^n).
\end{equation}
\end{lem}

The following diffotopy is due to \cite{MR1456322}. In the context of $\Cstar$-algebras a homotopy like this one is used to prove Bott periodicity and
the Pimsner-Voiculescu sequence.

\begin{lem}[Lemma 6.2 in \cite{MR1456322}]\label{ToepDiffo}
There is a unital diffotopy $\phi_t:\mcT\to \mcT\otimes_{\pi}\mcT$ such that
$$\phi_t(S)=S^2S^{*}\otimes 1 + f(t)(e\otimes S)+ g(t)(Se\otimes 1)$$
$$\phi_t(S^{*})=SS^{*2}\otimes 1 + \overline{f(t)}(e\otimes S^{*})+ \overline{g(t)}(eS^{*}\otimes 1)$$
where $f, g\in \CC[0, 1]$ are such that $f(0)=0$, $f(1)=1$, $g(0)=1$ and $g(1)=0$.
\end{lem}

Note that $\phi_0(S)=S\otimes 1$ and $\phi_1(S)=S^{2}S^{*}\otimes 1+e\otimes S$. Lemma \ref{universal_smooth_toeplitz} implies that, in order to define a
morphism from $\mcT$ to $\mcT\otimes_{\pi}\mcT$, we only need to check the relations on $S$ and $S^*$ and the bounds of (\ref{bounded_conditions_toeplitz}).

We finish this section with a result for tensoring algebras with a countable basis over $\CC$ equipped with the fine topology and the Toeplitz algebra.
 Although the result is known to experts, we give all the details since it allows us to prove Lemma \ref{algtensor}, which is a key ingredient in
 Proposition~\ref{Toeplitzext}, one of our main technical results.

\begin{lem}\label{lca_tensor_s}
The locally convex space $A\otimes_{\pi}s$ is isomorphic to the space $F$ of sequences $\{x_n\}_{n\in\NN}\subseteq A$ such that
$$||x||_{\rho, k}=\sum_{n\in\NN}|1+n|^k\rho(x(n))$$
is finite for all $k\in \NN$ and any continuous seminorm $\rho$ on $A$, where the topology on $F$ is defined by the seminorms $||\cdot||_{\rho,k}$.
\end{lem}
\begin{proof}
There is an inclusion $\phi:A\otimes s\to F$ defined by $a\otimes \alpha\in A\otimes s\mapsto \{x_n=\alpha_na\}\in F$. Let
$z=\sum_{t=1}^{N}a^{(t)}\otimes \alpha^{(t)}$ be an element of $A\otimes s$. We have
\begin{eqnarray*}
||\phi(z)||_{\rho, k} &=&  \sum_{n\in\NN} \rho\left(\sum_{t=1}^{N}a^{(t)}\alpha_n^{(t)}\right)|1+n|^{k}\\
&\leq& \sum_{n\in\NN} \sum_{t=1}^{N} \rho(a^{(t)})|\alpha_n^{(t)}||1+n|^{k}\\
&=& \sum_{t=1}^{N}\rho(a^{(t)})p_k(\alpha^{(t)}).
\end{eqnarray*}
This implies $||\phi(z)||_{\rho, k}\leq (\rho\otimes p_k)(z)$. We can write $z=\sum_{n\in\NN}\sum_{t=1}^{N}a^{(t)}\alpha_n^{(t)}\otimes e_n$ and therefore
\begin{eqnarray*}
(\rho\otimes p_{k})(z) &\leq& \sum_{n\in\NN} (\rho\otimes p_k)\left(\sum_{t=1}^{N}a^{(t)}\alpha_n^{(t)}\otimes e_n\right)\\
&=& \sum_{n\in\NN} \rho\left(\sum_{t=1}^{N}a^{(t)}\alpha_n^{(t)}\right)|1+n|^{k}\\
&=& ||\phi(z)||_{\rho, k}.
\end{eqnarray*}
This implies that $||\cdot||_{\rho, k}=\rho\otimes p_k$ in the image of $A\otimes s$.
Since all finite sequences in $A$ are in $A\otimes s$, $A\otimes s$ is dense in $F$. Since $F$ is a complete space, we conclude $A\otimes_{\pi}s=F$.
\end{proof}

\begin{lem}\label{algtensor}
Let $s$ be the locally convex space of rapidly decreasing sequences and $A$ an algebra with a countable basis over $\CC$ equipped with the fine topology. Then
 we have
$$A\otimes_{\pi}s=A\otimes s$$
as locally convex spaces. This implies that
$$A\otimes_{\pi} \mcT= A\otimes \mcT \quad \text{and} \quad A\otimes_{\pi} (\mcT \otimes_{\pi} \mcT)= A\otimes (\mcT \otimes_{\pi} \mcT)$$
as locally convex algebras.
\end{lem}

\begin{proof}
We prove that the space $F$ from Lemma \ref{lca_tensor_s} is equal to the algebraic tensor product $A\otimes s$. Let $\{v_n\}_{n\in \NN}$ be a countable
basis of $A$. Given $\{x_n\}_{n\in\NN}$ a sequence of elements in $A$ with $\rho_k(x)$ finite for all $k\in\NN$ we have, for $n$ fixed
$$x_n=\sum_{i\in \NN}\lambda_n^{(i)}v_i$$
where $\lambda_{n}^{(i)}\neq 0$ for finitely many $i\in \NN$.

First, we prove that span$\{x_n\}_{n\in \NN}$ is finite dimensional. Suppose this is not the case. We construct subsequences $\{x_{n_i}\}$ and $\{v_{m_i}\}$
such that $\lambda_{n_i}^{(m_i)}\neq 0$. Choose $n_1$ such that $x_{n_1}\neq 0$ and $m_1$ such that $\lambda_{n_1}^{(m_1)}\neq 0$. Suppose
$\{x_{n_1}, \dots, x_{n_k}\}$ and $\{v_{m_1}, \dots, v_{m_k}\}$ have been chosen. span$\{x_{i}\}_{i>n_k}$ is infinite dimensional and therefore
it is not contained in span$\{v_i\}_{1\leq i\leq m_k}$. Choose $n_{k+1}>n_k$ such that $x_{n_{k+1}}\notin\text{span} \{v_i\}_{1\leq i\leq m_k}$.
We can choose $m_{k+1}>m_k$ such that $\lambda_{n_{k+1}}^{(m_{k+1})}\neq 0$.

Now we define a seminorm in $A$
$$\rho\left(\sum_{i\in\NN}c_iv_i\right)=\sum_{i\in \NN}|c_i|\alpha_i$$
with $\alpha_i=0$ for $i\notin \{n_k\}_{k\in\NN}$ and $\alpha_{n_k}\geq |\lambda_{n_k}^{(m_k)}|^{-1}$. Thus we have $\rho(x_{n_k})\geq 1$ and
$$\rho_0(x)=\sum_{i\in\NN}\rho(x_i)\geq \sum_{i\in\NN}\rho(x_{n_i})$$
diverges. We conclude that  span$\{x_n\}_{n\in \NN}$ is finite dimensional.

Let $N\in\NN$ be such that span$\{x_n\}_{n\in \NN}\subseteq$ span$\{v_1, \dots, v_N\}$. That is
$$x_n=\sum_{i=0}^{N}\lambda_n^{(i)}v_i$$
Then
\begin{eqnarray*}
x &=&  \lim_{M\to \infty}\sum_{n=0}^{M} x_n\otimes e_n\\
&=& \lim_{M\to \infty}\sum_{n=0}^{M} \sum_{i=0}^{N}\lambda_n^{(i)}v_i\otimes e_n\\
&=& \lim_{M\to \infty}\sum_{i=0}^{N}v_i\otimes \sum_{n=0}^{M} \lambda_n^{(i)} e_n
\end{eqnarray*}
Consider the seminorm $p_j(\sum c_iv_i)=|c_{j}|$. Then, since $x\in A\otimes_{\pi}s$,
$$\sum_{n\in\NN}|1+n|^{k}p_i(x_n)=\sum_{n\in\NN}|1+n|^{k}|\lambda_{n}^{(i)}|<\infty$$
for all $k\in\NN$. Thus, for a fixed $i$, the sequences $\{\lambda_{n}^{(i)}\}$ are rapidly decreasing on $n$. Therefore
$\sum_{n\in \NN}\lambda_i^{(n)}e_n\in s$ and consequently $x=\sum_{i=1}^{N}v_i\otimes s_i\in A\otimes s$.

The equalities $A\otimes_{\pi} \mcT= A\otimes \mcT$ and $A\otimes_{\pi} (\mcT \otimes_{\pi} \mcT)= A\otimes (\mcT \otimes_{\pi} \mcT)$ follow because, by Lemma \ref{spaces_isomorphic_to_s}, we have $\mcT\cong s$ and $\mcT\otimes_{\pi} \mcT\cong s$ as
locally convex vector spaces.
\end{proof}


\section{Bivariant K-theory of locally convex algebras}
\label{definition of kk}


\subsection{Definition and properties of $kk^{alg}$}\

The bivariant $K$-theory $kk^{alg}$ is constructed by Cuntz in \cite{MR2240217}. In this section we give the definition of $kk^{alg}$ and state its main
properties. A complete treatise of these constructions in the context of bornological algebras can be found in \cite{MR2340673}. The proofs translate to the
context of locally convex algebras in a straightforward manner.

There is a canonical map
$$\langle J^kA, \mcK\otimes_{\pi}S^kB \rangle \to \langle J^{k+1}A, \mcK\otimes_{\pi}S^{k+1}B \rangle$$
that assigns to each morphism $\alpha$ the classifying map associated with the extension
$$
0\to \mcK\otimes_{\pi} S^{k+1}B\to \mcK\otimes_{\pi}CS^{k}B\to \mcK\otimes_{\pi}S^kB\to 0.
$$

\begin{defn}
Let $A$ and $B$ be locally convex algebras. We define
$$kk^{alg}(A, B)=\varinjlim_{k\in \NN} \langle J^{k}A, \mcK\otimes_{\pi} S^kB  \rangle$$
and for $n\in \ZZ$
$$kk_n^{alg}(A, B)=\varinjlim_{\substack{k\in \NN \\ k+n\geq 0}} \langle J^{k+n}A, \mcK\otimes_{\pi} S^kB  \rangle.$$
\end{defn}

The group structure of $kk_{n}^{alg}(A, B)$ is defined using Lemma \ref{group_structure}.

\begin{lem}
There is an associative product
$$kk_n^{alg}(A,B)\times kk_m^{alg}(B,C) \to kk_{n+m}^{alg}(A,C)$$
\end{lem}

\begin{proof}
Follows from Lemma 6.32 in \cite{MR2340673}.
\end{proof}

In view of this associative product we can regard locally convex algebras as objects of a category $\mathfrak{KK}^{alg}$ with morphisms between $A$ and
$B$ given by elements of $kk_{*}^{alg}(A, B)$.
Any morphism $\phi:A\to B$ of locally convex algebras induces an element $kk(\phi)\in kk^{alg}(A, B)$ which is associated with the diffotopy class of
$i\circ \phi:A\to B\to \mcK\otimes_{\pi} B$, where $i$ is the inclusion of $B$ into the first corner, i.e. $i(b)= e_{00}\otimes b$. We have
$kk(\phi)kk(\psi)=kk(\psi\circ \phi)$ (see Theorem 2.3.1 in \cite{MR2207702}) and therefore we have a functor
$$kk_{*}^{alg}:\mathfrak{lca}\to \mathfrak{KK}^{alg}.$$
In particular the identity of $A$ induces an element $kk(id_A)$ in $kk^{alg}(A, A)$ which is denoted by $1_A$.

\begin{defn}
A functor $F$ from the category of locally convex algebras to an abelian category $\mcC$ is called
\begin{enumerate}
\item diffotopy invariant if $F(f)=F(g)$ whenever $f$ and $g$ are diffotopic,
\item half exact for linearly split extensions if
$$F(A)\to F(B)\to F(C)$$
is exact whenever
$$0\to A\to B\to C\to 0$$
is a linearly split extension,
\item $\mcK$-stable if the natural inclusion $i:A\to \mcK\otimes_{\pi} A $, sending $a$ to $e_{00}\otimes a$ induces an isomorphism
$F(i):F(A)\to F(\mcK\otimes_{\pi}  A)$.
\end{enumerate}
\end{defn}

The functor $kk_{*}^{alg}:\mathfrak{lca}\to \mathfrak{KK}^{alg}$ is diffotopy invariant, half exact for linearly split extensions and is $\mcK$-stable.


\subsection{Bott Periodicity and Triangulated structure of $\mathfrak{KK}^{alg}$}\

The suspension of locally convex algebras determines a functor $S:\mathfrak{KK}^{alg}\to \mathfrak{KK}^{alg}$ with $S(A)=SA$.

\begin{thm}\label{bott_periodicity}[Bott periodicity]
There is a natural equivalence between $S^{2}$ and the identity functor, hence $\mathfrak{KK}^{alg}_{2n}(A,B)\cong \mathfrak{KK}^{alg}_{0}(A,B)$ and
$\mathfrak{KK}^{alg}_{2n+1}(A,B)\cong \mathfrak{KK}^{alg}_{1}(A,B)$.
\end{thm}

\begin{proof}
See Corollary 7.25 in \cite{MR2340673} and the discussion that follows.
\end{proof}

By Theorem \ref{bott_periodicity}, $S$ is an automorphism and $S^{-1}\cong S$. We recall the triangulated structure of $(\mathfrak{KK}, S)$.

Let $f:A\to B$ be a morphism in $\mathfrak{lca}$. The mapping cone of $f$ is defined as the locally convex algebra
$$C(f)=\{(x, g)\in A\oplus CB| f(x)=g(0)\}.$$
The triangle
$$
\xymatrix{
SB\ar[r]^{kk(\iota)} & C(f)\ar[r]^{kk(\pi)} & A\ar[r]^{kk(f)} & B
}
$$
in $(\mathfrak{KK}^{alg}, S)$, where $\pi:C(f)\to A$ is the projection into the first component and $\iota:SB\to C(f)$ is the inclusion into the first
component, is called a mapping cone triangle.

Let $E: 0\to A \overset{f}\to  B\overset{g}\to C \to 0$
 be a linearly split extension in $\mathfrak{lca}$. This induces an element $kk(E)\in kk_{1}^{alg}(C, A)$ that corresponds to the classifying map $JC\to A$
 of the extension and hence an element $kk(E)\in kk^{alg}(SC, A)$. The triangle
$$
\xymatrix{
SC\ar[r]^{kk(E)} & A\ar[r]^{kk(f)} & B\ar[r]^{kk(g)} & C
}
$$
 in $(\mathfrak{KK}^{alg}, S)$ is called an extension triangle.

\begin{prop}\label{triangulated}
The category $\mathfrak{KK}^{alg}$ with suspension automorphism $S:\mathfrak{KK}^{alg}\to \mathfrak{KK}^{alg}$ and with triangles isomorphic to mapping cone
triangles as exact triangles is a triangulated category. Furthermore, extension triangles are exact.
\end{prop}

\begin{proof}
See Propositions 7.22 and  7.23 in \cite{MR2340673}.
\end{proof}


\subsection{Stabilization by Schatten ideals}\label{kk-schatten-ideals}\

In \cite{MR2207702}, Cuntz and Thom define a related bivariant $K$-theory in the category $\mathfrak{lca}$. We recall the definition for the case of the
Schatten ideals. Let $\HH$ denote an infinite dimensional separable Hilbert Space.

\begin{defn}
The Schatten ideals $\mathcal{L}_p\subseteq B(\HH)$, for $p\geq 1$, are defined by
$$\mathcal{L}_p=\{x\in B(\HH)| \ Tr |x|^p<\infty \}.$$
Equivalently, $\mathcal{L}_p$ consists of the space of bounded operators such that the sequence of its singular values $\{\mu_n\}$ is in $l^{p}(\NN)$.
\end{defn}

\begin{defn}\label{def kk Lp}
Let $A$ and $B$ be locally convex algebras and $p\geq 1$. We define
$$kk^{\mathcal{L}_p}_n(A, B)=kk^{alg}(A, B\otimes_{\pi}\mathcal{L}_p).$$
\end{defn}

The groups $kk^{\mathcal{L}_p}(A, B)$, for all $p\geq 1$, are isomorphic (Corollary 2.3.5 of \cite{MR2207702}). Moreover, when $p>1$,
this bivariant $K$-theory is related to algebraic $K$-theory in the following manner:

\begin{thm}[Theorem 6.2.1 in \cite{MR2207702}]
For every locally convex algebra $A$ and $p>1$ we have
$$kk^{\mathcal{L}_p}_0(\CC, A)=K_0(A\otimes_{\pi}\mathcal{L}_p).$$
\end{thm}

\begin{cor}[Corollary 6.2.3 in \cite{MR2207702}]
The coefficient ring $kk^{\mathcal{L}_p}_*(\CC, \CC)$ is isomorphic to $\ZZ[u, u^{-1}]$ with $deg(u)=2$.
\end{cor}

This implies that  $kk^{\mathcal{L}_p}_0(\CC, \CC)=\ZZ$ and  $kk^{\mathcal{L}_p}_1(\CC, \CC)=0$.

Consider the category $\mathfrak{KK}^{\mathcal{L}_p}$ whose objects are locally convex algebras and whose morphisms are
given by the graded groups $kk_{*}^{\mathcal{L}_p}(A,B)$.
Since $kk_{*}^{\mathcal{L}_p}:\mathfrak{lca}\to \mathfrak{KK}^{\mathcal{L}_p}$ is diffotopy invariant, half exact for linearly split extensions and is $\mcK$-stable (see Lemma 7.20 in \cite{MR2340673}), by the universal property of $kk_{*}^{alg}$ we have a functor $\mathfrak{KK}^{alg}\to \mathfrak{KK}^{\mathcal{L}_p}$.


\subsection{Weak Morita equivalence}\

In the context of separable $\Cstar$-algebras, two algebras $A$ and $B$ are strong Morita equivalent if and only if $\KK\otimes A\cong \KK\otimes B$ (they are
 stably isomorphic). Therefore a strong Morita equivalence of separable $\Cstar$-algebras induces an equivalence in $KK$. In the case of locally convex
 algebras we recall the definition of weak Morita equivalence from \cite{MR2240217}, which still give us an isomorphism between two objects in $\mathfrak{KK}^{alg}$.

A Morita context yields the data required in order to define maps $A\to \mcK\otimes_{\pi} B$.

\begin{defn}\label{MoritaDef}
Let $A$ and $B$ be locally convex algebras. A Morita context from $A$ to $B$ consists of a locally convex algebra $E$ that contains $A$ and $B$ as subalgebras
and two sequences $(\xi_i)_{i\in\NN}$ and $(\eta_j)_{j\in \NN}$ of elements of $E$ that satisfy
\begin{enumerate}
\item $\eta_jA\xi_i\subset B$ for all $i, j$.
\item The sequence $(\eta_ja\xi_i)$ is rapidly decreasing for each $a\in A$. That is, for each continuous seminorm $\alpha$ in $B$, $\alpha(\eta_ja\xi_i)$
is rapidly decreasing in $i, j$.
\item For all $a\in A$, $(\sum \xi_i\eta_i)a=a$.
\end{enumerate}
\end{defn}
A Morita context $((\xi_i), (\eta_j))$ from $A$ to $B$ determines a homomorphism $A\to \mcK \otimes_{\pi} B$ defined by
 $a\mapsto \sum_{i,j\in\NN}e_{ij}\otimes \eta_j a\xi_i $. Thus it determines an element $kk((\xi_i), (\eta_j))$ of $kk_0^{alg}(A, B)$.

In the next proposition, we give conditions for a Morita context to determine an equivalence in $\mathfrak{KK}^{alg}$.

\begin{prop}\label{MoritaCond}
Let $((\xi_i), (\eta_j))$ be a Morita context from $A$ to $B$ in $E$. If $((\xi_l'), (\eta_k'))$ is a Morita context from $B$ to $A$ in the same locally convex
algebra and if $A\xi_i\xi_l'\subset A$ and $\eta_k'\eta_j A\subset A$ for all $i, j, k, l$; then
$$kk((\xi_i), (\eta_j))\cdot kk((\xi_l'), (\eta_k'))=1_A.$$
Therefore, if we also have $B\xi_l'\xi_i\subset B$ and $\eta_k\eta_j'B\subset B$ for all $i, j, k, l$, then $kk((\xi_i), (\eta_j))$ is invertible in
$kk^{alg}$.
\end{prop}

\begin{proof}
See Lemma 7.2 in \cite{MR2240217}.
\end{proof}


\subsection{Quasihomomorphisms}\

The definition of a quasihomomorphism goes back to \cite{MR733641}. We give the definition of a quasihomomorphisms in the context of locally convex algebras
from \cite{MR2207702}. A quasihomomorphism from $A$ to $B$ determines an element in $kk^{alg}(A,B)$. As a matter of fact it determines a morphism from $E(A)$
to $E(B)$  for any split-exact functor
$E:\mathfrak{lca}\to \mathfrak{C}$ where $\mathfrak{C}$ is an additive category. The reader can also see Section 4 in \cite{MR3054304} and
Section 3.3.1 in \cite{MR2340673}.
\begin{defn}
Let $A$, $B$ and $D$ be locally convex algebras with $B$ a closed subalgebra of $D$. A quasihomomorphism from $A$ to $B$ in $D$ is a pair of homomorphisms
$(\alpha, \bar\alpha)$ from $A$ to $D$ such that $\alpha(x)-\bar\alpha(x)\in B$, $\alpha(x)B\subset B$ and $B\alpha(x)\subset B$ for all $x\in A$. We denote
such quasihomomorphism by $(\alpha, \bar\alpha): A \rightrightarrows D\triangleright B$.
\end{defn}

The original definition of quasihomomorphisms required $B$ to be an ideal in $D$ (Definition 2.1 in \cite{MR733641}). Note that if $B$ is an ideal then the
conditions $\alpha(x)B\subset B$ and $B\alpha(x)\subset B$ are satisfied automatically. On the other hand we only need to check these conditions in a set of
algebraic generators of $A$.

\begin{rmk}
Let $G\subset A$ be a subset that generates $A$ as a locally convex algebra. If $\alpha(x)-\bar\alpha(x)\in B$, $\alpha(x)B\subset B$ and $B\alpha(x)\subset B$
 for all $x\in G$, then the conditions are also satisfied for all $x\in A$.
\end{rmk}

Next we see how a quasihomomorphism $(\alpha, \alpha'): A \rightrightarrows D\triangleright B$ determines an element
$kk(\alpha, \bar\alpha)\in kk^{alg}(A, B)$. As a matter of fact, we work with split exact functors from $\mathfrak{lca}$ to an additive category
$\mathfrak{C}$. An extension
$0\to A \to B \overset{\pi}\to C \to 0$
in $\mathfrak{lca}$  is split if there is a morphism of locally convex algebras $s:C\to B$ such that $\pi s=id_{C}$.

\begin{defn}
Let $\mathfrak{C}$ be an additive category. A sequence $A\to B\to C$ in $\mathfrak{C}$ is split exact if it is isomorphic to the sequence $A\to A\oplus C\to C$
with the natural inclusion and projection. A functor $E:\mathfrak{lca}\to \mathfrak{C}$ is called split exact if it sends split extensions in $\mathfrak{lca}$
 to split exact sequences in $\mathfrak{C}$.
\end{defn}

\begin{lem}\label{def-quasi-homo}[Section 3.2 in \cite{MR2207702}]
Let $E$ be a split exact functor from $\mathfrak{lca}$ to an additive category $\mathfrak{C}$. Then a quasihomomorphism
$(\alpha, \alpha'): A \rightrightarrows D\triangleright B$
determines a morphism $E(\alpha, \bar\alpha):E(A)\to E(B)$ in $\mathfrak{C}$.
\end{lem}

\begin{proof}
Let $D'$ be the closed subalgebra of $A\oplus D$ generated by all elements $(a, \alpha(a))$ and $(0, b)$ with $a\in A$ and $b\in B$. Then we have an exact
 sequence
$$0\to B\to D'\to A\to 0$$
with the inclusion $B\subseteq D'$ given by $b\mapsto (0, b)$ and the projection $\pi:D'\to A$ defined by $\pi(a, x)=a$. This extension has two splits
$\alpha', \bar\alpha':A\to D'$ defined by $\alpha'(a)=(a, \alpha(a))$ and $\bar\alpha'(a)=(a, \bar\alpha(a))$. Because of the split-exactness of
$E$, $E(B)\to E(D')$ is a kernel of $E(D')\to E(A)$.  Therefore, the morphism $E(\alpha')-E(\bar\alpha'):E(A)\to E(D')$ defines a morphism
$E(\alpha, \bar\alpha):E(A)\to E(B)$.
\end{proof}

The following proposition summarizes some properties of quasihomomorphisms.

\begin{prop}\label{quasi-homo}
Let $E$ be a split exact functor from $\mathfrak{lca}$ to an additive category $\mathfrak{C}$ and $(\alpha, \alpha'): A \rightrightarrows D\triangleright B$ be
a quasihomomorphism from $A$ to $B$ in $D$. We have
\begin{enumerate}

\item $E(\alpha, \bar\alpha)=-E(\bar\alpha, \alpha)$.

\item Let $\phi=\alpha-\bar\alpha$. If $\phi(x)\bar\alpha(y)=\bar\alpha(y)\phi(x)=0$ for all $x, y\in A$, then $\phi$ is a homomorphism and
$E(\alpha, \bar\alpha)=E(\phi)$.

\item For any morphism $\phi: A'\to A$, $(\alpha\circ \phi, \bar\alpha \circ \phi):A'\to B$ is a quasihomomorphism from $A'$ to $B$ in $D$ and
$$E(\alpha\circ \phi, \bar\alpha \circ \phi)=E(\alpha, \bar\alpha)\circ E(\phi).$$

\item If $\psi: D \to F$ is a morphism such that $\psi|_B:B \to C\subset F$ and the morphisms $\psi\circ \alpha, \psi\circ \bar\alpha:A\to F$ define a
quasihomomorphism from $A$ to $C$ in $F$, then
$$E(\psi\circ \alpha, \psi\circ \bar\alpha)=E(\psi|_B)\circ E(\alpha, \bar\alpha).$$

\item Let $\alpha$ and $\bar\alpha$ be homomorphisms from $A$ to $D[0, 1]$ such that $\alpha(x)-\bar\alpha(x)\in B[0, 1]$, $\alpha(x)B[0, 1]\subset B[0, 1]$
and $B[0, 1]\alpha(x)\subset B[0, 1]$ for all $x\in A$. If $E$ is diffotopy invariant, then $E(\alpha_1, \bar\alpha_1)=E(\alpha_0, \bar\alpha_0)$ (where
 $\alpha_t=ev_t\circ \alpha$).
\end{enumerate}
\end{prop}

\begin{proof}
For $(1)-(4)$ see Proposition 21 in \cite{MR2964680}.

To prove $(5)$, we consider the evaluation maps $ev_t:D[0, 1]\to D$. They restrict to the evaluation maps $ev_t: B[0, 1]\to B$. To apply $(3)$ we need to check
that the morphisms $ev_t\circ \alpha, ev_t\circ \bar\alpha: A\to D$ define a quasihomomorphism from $A$ to $B$ in $D$. First notice that
$(ev_t\circ \alpha)(a)-(ev_t\circ\bar\alpha)(a)=(ev_t\circ (\alpha-\bar\alpha))(a)$ is in $B$ because $(\alpha-\bar\alpha)(a)\in B[0, 1]$.
 Now consider an element $b\in B$. We want to prove that the product $(ev_t\circ \alpha)(a)b$ is in $B$. Consider a function $\phi\in B[0, 1]$ such that
 $ev_t\circ f=b$. Then $(ev_t\circ \alpha)(a)b=ev_t\circ (\alpha(a)f)$ and $\alpha(a)f\in B[0, 1]$. Similarly, we can prove that
 $B(ev_t\circ \alpha)(a)\subseteq B$.

We can now apply $(3)$ and we obtain $E(ev_t\circ \alpha, ev_t\circ \bar\alpha)=E(ev_t)\circ E(\alpha, \bar\alpha)$. Since $E$ is diffotopy invariant,
we have $E(ev_0)=E(ev_1)$, which concludes the result.
\end{proof}


\section{Generalized Weyl algebras}
\label{generalized Weyl algebra}

Generalized Weyl algebras were introduced by Bavula (see \cite{MR1804517}) and have been amply studied. Examples of generalized Weyl algebras include the Weyl
 algebra, the quantum Weyl algebra, the quantum plane, the enveloping algebra of $\mathfrak{sl}_2$ $U(\mathfrak{sl_2})$, its primitive factors
 $B_\lambda=U(\mathfrak{sl}_2)/\langle C-\lambda \rangle$ where $C$ is the Casimir element (see Example 4.7 in \cite{MR1247356}) and the quantum weighted projective lines $\mathcal{O}(\mathbb{WP}_q(k, l))$ (see \cite{MR2989456}).

In our context, generalized Weyl algebras provide a family of examples of $\ZZ$-graded algebras that are smooth generalized crossed products and do not satisfy
 the condition of being tame smooth (and therefore they are outside the framework of \cite{MR3054304}).

\begin{defn}\label{GWA} Let $D$ be a ring, $\sigma \in Aut(D)$ and
$a$ a central element of $D$. The generalized Weyl algebra $D(\sigma, a)$ is the algebra generated by $x$ and $y$ over $D$ satisfying
\begin{equation}\label{relations_GWA}
xd=\sigma(d)x,\  yd=\sigma^{-1}(d)y,\ yx=a \text{ and } xy=\sigma(a)
\end{equation}
for all $d\in D$.
\end{defn}

\begin{exs}
The following are examples of generalized Weyl algebras:
\begin{enumerate}
\item The Weyl algebra
$$A_1(\CC)=\CC\langle x, y | xy-yx=1\rangle$$
is isomorphic to $\CC[h](\sigma, h)$, with $\sigma(h)=h-1$.
\item The quantum Weyl algebra
$$A_q(\CC)=\CC\langle x, y | xy-qyx=1\rangle$$
is isomorphic to $\CC[h](\sigma, h-1)$, with $\sigma(h)=qh$.
\item The quantum plane
$$\CC\langle x, y |xy=qyx\rangle$$
is isomorphic to $\CC[h](\sigma, h)$, with $\sigma(h)=qh$.
\item The primitive quotients of $U(\mathfrak{sl}_2)$ (see Example 3.2 in \cite{MR1804517}),
$$B_{\lambda}=U(\mathfrak{sl}_2)/\langle c-\lambda\rangle, \quad \lambda\in\CC,$$
are isomorphic to $\CC[h](\sigma, P)$, with $\sigma(h)=h-1$ and $P(h)=-h(h+1)-\lambda/4$.
\item The quantum weighted projective line  or the quantum spindle algebra $\mathcal{O}(\mathbb{WP}_q(k, l))$ is
 isomorphic to $\CC[h](\sigma, P)$ with $P(h)=h^{k}\prod_{i=0}^{l-1}(1-q^{-2i}h)$ and $\sigma(h)=q^{2l}h$ (see Theorem 2.1 in \cite{MR2989456} and Example 3.8 in \cite{MR3465890}).
\end{enumerate}
\end{exs}

\begin{lem}\label{GWA-grading}
A generalized Weyl algebra has a $\ZZ$-grading $A=\bigoplus_{n\in \ZZ} A_n$ where $A_0=D$ and
\begin{equation}\label{homogeneous_subspaces_GWA}
A_n=
\begin{cases}
Dy^n & n>0 \\
Dx^n & n<0.
\end{cases}
\end{equation}
\end{lem}

\begin{proof}
Consider the grading in $A=D(\sigma, a)$ defined by setting the degree of $y$ equal to $1$, the degree of $x$ equal to $-1$, and the degree of all elements of
$D$ equal to $0$. That is, the degree of the monomial $\prod_{i=1}^{n} d_ix^{\alpha_i}y^{\beta_i}$, with $d_i\in D$, is equal to
$\sum_{i=1}^{n}\beta_i-\sum_{i=1}^{n}\alpha_i$. Since the relations defining $A$ are compatible with the grading, the algebra $A$ is $\ZZ$-graded.

Now consider the following relations in $A$. We have
\begin{eqnarray*}
x^{n}y^{n} &=& \sigma^{n}(a)\sigma^{n-1}(a)\dots \sigma(a) \\
y^{n}x^{n} &=& \sigma^{-(n-1)}(a)\sigma^{-(n-2)}(a)\dots a
\end{eqnarray*}
Using induction on the length of the monomial $\prod_{i=1}^{n} d_ix^{\alpha_i}y^{\beta_i}$ we prove (\ref{homogeneous_subspaces_GWA}). Note that
$Dy^{n}=y^{n}D$ and $Dx^{n}=x^{n}D$.
\end{proof}

In the case of generalized Weyl algebras over $\CC[h]$, we have the following result.

\begin{cor}\label{GWA_countable_basis}
The generalized Weyl algebra $A=\CC[h](\sigma, P)$, with $P\in \CC[h]$, has a countable basis over $\CC$.
\end{cor}

\begin{proof}
A basis is given by the elements $h^{n}$, $h^{n}y^{m}$ and $h^{n}x^{m}$ for $n\in \NN$, $m\geq 1$.
\end{proof}

There are several ways of writing the same generalized Weyl algebra. The conjugation of $\sigma$ by an automorphism $\tau$ of $D$ gives rise to an isomorphism
 of generalized Weyl algebras.

\begin{lem}\label{conjugate_automorphism}
Let $\sigma$, $\tau$ be automorphisms of $D$ and let $a$ be a central element of $D$. Then $\tau(a)$ is central in $D$ and
$$D(\sigma, a)\cong D(\tau\sigma\tau^{-1}, \tau(a)).$$
\end{lem}

\begin{proof}
Let $x'$ and $y'$ be the generators of $D(\tau\sigma\tau^{-1}, \tau(a))$ over $D$. There is a morphism $\phi: D(\sigma, a)\to D(\tau\sigma\tau^{-1}, \tau(a))$
defined by $x\mapsto x'$, $y\mapsto y'$, $d\mapsto \tau(d)$, for all $d\in D$. We need to check that $\phi$ is compatible with the relations
(\ref{relations_GWA}). Using the relations defining $D(\tau\sigma\tau^{-1}, \tau(a))$ we have
\begin{eqnarray*}
x'\tau(d) &=& (\tau\sigma\tau^{-1})(\tau(d))x'=\tau(\sigma(d))x' \\
y'\tau(d) &=& (\tau\sigma^{-1}\tau)(\tau(d))y'=\tau(\sigma^{-1}(d))y'\\
x'y'&=&\tau(a)\\
y'x'&=& (\tau\sigma\tau^{-1})(\tau(a))=\tau(\sigma(a)).
\end{eqnarray*}
$\phi^{-1}$ is defined by $x'\mapsto x$, $y'\mapsto y$, $d\mapsto \tau^{-1}(d)$ for all $d\in D$.
\end{proof}

In the case $D=\CC[h]$, we use Lemma \ref{conjugate_automorphism} to write a given generalized Weyl algebra in a canonical form. Any automorphism of $\CC[h]$,
 is of the form $\sigma(h)=qh+h_0$ with $q, h_0\in \CC$ and $q\neq 0$. We have three cases
\begin{enumerate}
\item $\sigma$ is conjugate to $id$ if and only if $\sigma=id$,
\item if $q=1$ and $h_0\neq 0$, then $\sigma$ is conjugate to $h\mapsto h-1$,
\item if $q\neq 1$, then $\sigma$ is conjugate to $h\mapsto qh$.
\end{enumerate}

Combining this with Lemma \ref{conjugate_automorphism}, we obtain the following result.

\begin{prop}[Compare with Proposition 2.1.1 in \cite{MR2235811}.]\label{classical_and_quantum}
Let $A=\CC[h](\sigma, P)$, with $P\in \CC[h]$ and $\sigma(h)=qh+h_0$ with $q, h_0\in\CC$ and $q\neq 0$. The following facts hold:
\begin{enumerate}
\item If $\sigma=id$, then $A\cong \CC[h, x, y]/(yx-P)$.
\item If $q=1$ and $h_0\neq 0$ then $A\cong \CC[h](\sigma_1, P_1)$ with $\sigma_1(h)=h-1$ and $P_1(h)=P(-h_0h)$.
\item If $q\neq 1$ then $A\cong \CC[h](\sigma_1, P_1)$ with $\sigma_1(h)=qh$ and $P_1(h)=P(h+\frac{h_0}{1-q})$.\qed
\end{enumerate}
\end{prop}
\begin{proof}
(1) is straightforward, for~(2) and~(3) use Lemma~\ref{conjugate_automorphism} with $\tau(h)=-h_0h$ and $\tau(h)=h+\frac{h_0}{1-q}$, respectively.
\end{proof}

By Proposition \ref{classical_and_quantum}, we can assume that $\sigma=id$, $\sigma(h)=h-1$ or $\sigma(h)=qh$ for some $q\neq 0$.

\begin{prop}\label{canonical_GWA}
Let $A=\CC[h](\sigma, P)$, with $P\in \CC[h]$. The following facts hold:
\begin{enumerate}
\item If $\sigma(h)=h-1$ and $P$ is a non constant polynomial, then $A\cong \CC[h](\sigma, P_1)$ with $P_1(0)=0$.
\item If $\sigma(h)=qh$ and $P$ has a nonzero root, then $A\cong \CC[h](\sigma, P_1)$ with $P_1(1)=0$.
\end{enumerate}
\end{prop}

\begin{proof}
Follows from Lemma \ref{conjugate_automorphism}. In~(1) we set $\tau(h)=h-\lambda$ for any root $\lambda$ of $P$, and in~(2) we set
$\tau(h)=\lambda h$, where $\lambda$ is a non zero root of $P$.
\end{proof}

Note that $P_1$ in Proposition~\ref{classical_and_quantum} has a non zero root if and only if $P$ has a root different from $\frac{h_0}{1-q}$.

It is worth mentioning that generalized Weyl algebras over $\CC[h]$ have been classified up to isomorphism (see \cite{MR1804517}and \cite{MR2235811}).

To finish this section, we construct faithful representations for the generalized Weyl algebras covered in cases $(1)$ and $(2)$ of Proposition \ref{canonical_GWA}. We define
$V_{\NN}$ as the vector space of sequences of complex numbers indexed by $\NN$. Let $\mcU_1, \mcU_{-1}\in End(V_{\NN})$ be the shift to the right and the
shift to the left respectively. Note that $\mcU_{-1}\mcU_{1}=1$, $\mcU_{1}\mcU_{-1}=1-e_{00}$.
$$
\mcU_{1}=
\begin{bmatrix}
0 & 0 & 0 &\cdots \\
1 & 0 & 0 &  \\
0 & 1 & 0 &   \\
0 & 0 & 1 &  \\
\vdots & & & \ddots
\end{bmatrix}\qquad
\mcU_{-1}=
\begin{bmatrix}
0 & 1 & 0 & 0 &  \cdots \\
0 & 0 & 1 &  0 &  \\
0 & 0 & 0 &  1 &   \\
\vdots &  & & & \ddots
\end{bmatrix}
$$

Additionally, we use the following elements $N=\sum_{i\in\NN} (-i)e_{i,i}$ and $G=\sum_{i\in\NN} q^ie_{i,i}$ for $q\neq 0$ not a root of unity.
$$
N=
\begin{bmatrix}
0 & 0 & 0 & 0 & \cdots \\
0 & -1 & 0 &  0 &  \\
0 & 0 & -2 & 0 &  \\
0 & 0 & 0 & -3 \\
\vdots & & & & \ddots
\end{bmatrix}\qquad
G=
\begin{bmatrix}
1 & 0 & 0 & 0 & \cdots \\
0 & q & 0 &  0 &  \\
0 & 0 & q^2 & 0 &  \\
0 & 0 & 0 & q^3 \\
\vdots & & & & \ddots
\end{bmatrix}
$$

\begin{lem}\label{relations-N-and-G}
The following relations are satisfied in $End(V_{\NN})$.
\begin{enumerate}
\item $\mcU_1N=(N+1)\mcU_1$,
\item $\mcU_{-1}N=(N-1)\mcU_{-1}$,
\item $\mcU_1G=(q^{-1}G)\mcU_1$,
\item $\mcU_{-1}G=(qG)\mcU_{-1}$. \qed
\end{enumerate}
\end{lem}

As a consequence of Lemma \ref{relations-N-and-G}, we obtain that the subalgebras $\mcE_1$ and $\mcE_2$ of $End(V_{\NN})$ generated by $\{\mcU_{1}, \mcU_{-1}, N\}$ and $\{\mcU_{1}, \mcU_{-1}, G\}$, respectively, have countable basis over $\CC$ and therefore they are locally convex algebras with the fine topology.

\begin{lem}\label{representation}
We have the following representations for generalized Weyl algebras $A=\CC[h](\sigma, P(h))$ in the following cases.
\begin{enumerate}
\item If $\sigma(h)=h-1$ and $P$ is a nonzero polynomial with $P(0)=0$, then there is a faithful representation $\rho: A \to \mcE_1$ such that
$$\rho(h)=N,\  \rho(x)=\mathcal{U}_{-1} \text{ and }\rho(y)=P(N)\mcU_1=\mcU_1P(N-1)$$
\item If $\sigma(h)=qh$ with $q\neq 0$ not a root of unity and $P(1)=0$, then there is a faithful representation $\rho: A \to \mcE_2$ such that
$$\rho(h)=G, \ \rho(x)=\mathcal{U}_{-1} \text{ and }\rho(y)=P(G)\mcU_1=\mcU_1P(qG)$$
\end{enumerate}
\end{lem}

\begin{proof}
For $(1)$, first we notice that we have an injective homomorphism $\CC[h]\hookrightarrow End(V_{\NN})$ defined by $h\mapsto N$.

This homomorphism is injective because all the entries in the diagonal of matrix $N$ are different. With $P(0)=0$ we will see that the relations of
$\CC[h](\sigma, P(h))$ hold. To prove this, we use the relations of Lemma \ref{relations-N-and-G}. For a polynomial $\alpha(h)\in\CC[h]$ we have
\begin{eqnarray*}
\rho(x\alpha(h)) &=& \mcU_{-1}\alpha(N) = \alpha(N-1)\mcU_{-1} = \rho(\alpha(h-1)x) \\
\rho(y\alpha(h)) &=& P(N)\mcU_{1}\alpha(N)=\alpha(N+1)P(N)\mcU_{1}=\rho(\alpha(h+1)y)\\
\rho(yx) &=& \mcU_{1}P(N-1)\mcU_{-1}=\mcU_{1}\mcU_{-1}P(N)=(1-e_{00})P(N)=P(N)=\rho(P(h))\\
\rho(xy) &=& \mcU_{-1}\mcU_{1}P(N-1)=P(N-1)=\rho(P(h-1))
\end{eqnarray*}
We use that $P(0)=0$ in the third row to guarantee $(1-e_{00})P(N)=P(N)$.

Now we prove that $\rho$ is injective. Let
$$\alpha=\sum_{n\geq 0}p_n(h)y^n+\sum_{m<0}q_m(h)x^m$$
be an element of $A$. Then we have
$$\rho(\alpha)=\sum_{n\geq 0}p_n(P(N))(P(N)\mcU_1)^n+\sum_{m<0}q_m(P(N))\mcU_{-1}^m.$$
Note that $(P(N)\mcU_1)^{n}=Q_n(N)\mcU_1^n$ where $Q_n(N)=P(N)P(N+1)\dots P(N+(n-1))$. Therefore if $\rho(\alpha)=0$ then $q_m=0$ and because $Q_n\neq 0$,
we have $p_n=0$. Therefore $\alpha=0$ and so $\rho$ is injective.

$(2)$ is proved in a similar way: we have an injective homomorphism $\CC[h]\hookrightarrow End(V_{\NN})$ defined by $h\mapsto G$. This homomorphism is
injective because $q\neq 0$ not a root of unity imply that all the entries in the diagonal of matrix $G$ are different. Using $P(1)=0$, it is easy to see
that the relations of $D(\sigma, a)$ hold. We also need to use the relations of Lemma \ref{relations-N-and-G}. We prove that $\rho$ is injective in a
similar way. In this case we note that $(P(G)\mcU_1)^{n}=Q_n(G)\mcU_1^n$ where $Q_n(G)=P(G)P(q^{-1}G)\dots P(q^{-(n-1)}G)$.

\end{proof}


\section{$kk^{alg}$ invariants of generalized Weyl algebras}
\label{computations}
In this section, we compute the isomorphism class  in $\mathfrak{KK}^{alg}$ of generalized Weyl algebras $A=\CC[h](\sigma, P)$ where $\sigma(h)=qh+h_0$ is an
automorphism of $\CC[h]$ and $P\in \CC[h]$. We summarize our results in the table below.

\begin{table}[htb]
\begin{tabular}{| p{3.1cm} |  l |  l | p{1.8cm} | l |}
\hline
\multicolumn{2}{|c|}{Conditions} &  \multicolumn{2}{|c|}{Results}  &  Observation \\ \hline
\multirow{2}{*}{$P$ is constant} &  $P=0$  & $A\cong_{\mathfrak{KK}^{alg}}\CC$ & Prop  \ref{case1}& $A$ $\NN$-graded\\ \cline{2-5}
  & $P\neq 0$ & $A\cong_{\mathfrak{KK}^{alg}}S\CC \oplus \CC$ & Prop \ref{invariants-exceptional-case-1}& $A$ tame smooth\\ \hline
\multirow{3}{3.1cm}{$P$ is nonconstant with $r$ distinct roots}
 & $q$ not a root of unity & $A\cong_{\mathfrak{KK}^{alg}}\CC^{r}$ & Thm \ref{main_result} Prop \ref{case4} &\\ \cline{2-5}
 & $q=1$ and $h_0\neq 0$ & $A\cong_{\mathfrak{KK}^{alg}}\CC^{r}$ & Thm \ref{main_result} &\\ \hline
\end{tabular}
\end{table}

In Section \ref{toeplitz_SGCP}, we recall the definition of $\mcT_B$ from \cite{MR3054304}. Generalized Weyl algebras $A=\CC[h](\sigma, P)$ are smooth
generalized crossed products and in Proposition~\ref{Toeplitzext} we construct a linearly split extension
$$0\to \Lambda_A\to \mcT_A\to A \to 0.$$
In the case where $P$ is a non constant polynomial, $A$ is a generalized crossed product that is not tame smooth so we cannot apply the results of
\cite{MR3054304} directly.  In this case we follow the methods of \cite{MR2240217} and \cite{MR3054304} to obtain
$$
\Lambda_A\cong_{\mathfrak{KK}^{alg}} A_1A_{-1}\quad\text{(Theorem~\ref{moritaker} )}\qquad \text{and}
\quad \mcT_A\cong_{\mathfrak{KK}^{alg}} A_0\quad\text{(Theorem~\ref{toep_grado0} )}
$$
 in the cases where $P$ is non constant and
\begin{itemize}
\item $q=1$ and $h_0\neq 0$ or
\item $q$ is not a root of unity and $P$ has a root different from $\frac{h_0}{1-q}$.
\end{itemize}
With these isomorphisms we construct in Theorem~\ref{exact_triangle} an exact triangle
$$
SA \to A_1A_{-1}\overset{0}\to A_0\to A
$$
in the triangulated category $(\mathfrak{KK}^{alg}, S)$ (see Proposition \ref{triangulated}). This implies  $$A=A_0\oplus S(A_1A_{-1}).$$

In Proposition~\ref{invariants_A1A-1}, we prove that $A_1A_{-1}\cong_{\mathfrak{KK}^{alg}} S\CC^{r-1}$, and since by Lemma~\ref{N-graded}
we know that $A_0\cong_{\mathfrak{KK}^{alg}}\CC$, we obtain our main result Theorem~\ref{main_result}: in these cases $A\cong_{\mathfrak{KK}^{alg}} \CC^{r}$.

We also determine the ${\mathfrak{KK}^{alg}}$-class of $A$ when $A$ is $\NN$-graded. In this case Lemma \ref{N-graded} gives us
$A\cong_{\mathfrak{KK}^{alg}} A_0$. This is the case when
\begin{itemize}
\item $P$ is nonconstant, $q$ is not a root of unity and $P$ has only $\frac{h_0}{1-q}$ as a root or
\item $P=0$.
\end{itemize}
In both cases we obtain $A\cong_{\mathfrak{KK}^{alg}}\CC$ in Propositions~\ref{case1} and~\ref{case4}.

If $P$ is a nonzero constant polynomial, $A$ is a tame smooth generalized crossed product and the results from \cite{MR3054304} apply. In this case
we obtain $A\cong_{\mathfrak{KK}^{alg}}  S\CC\oplus\CC$ in Proposition~\ref{invariants-exceptional-case-1}.


\subsection{The Toeplitz algebra of a smooth generalized crossed product}\label{toeplitz_SGCP}\

In \cite{MR3054304}, Gabriel and Grensing define smooth generalized crossed products. These are involutive locally convex algebras analog to $\Cstar$-algebra
 generalized crossed products in~\cite{MR1467459}. In the same article~\cite{MR3054304}, sequences analog to the Pimsner-Voiculescu exact sequence are
 constructed for smooth generalized crossed products that are tame smooth (see definition 18 in \cite{MR3054304}).

\begin{defn}
A gauge action $\gamma$ on a locally convex algebra $B$ is a pointwise continuous action of $S^1$ on $B$. An element $b\in B$ is called gauge smooth if the
map $t\mapsto \gamma_t(b)$ is smooth.
\end{defn}

If we have a gauge action on $B$, then
$B_n=\{b\in B| \gamma_t(b)=t^n b,\ \forall t\in S^1\}$ defines a natural $\ZZ$-grading of $B$.

\begin{defn}\label{GCP_definition}
A smooth generalized crossed product is a locally convex algebra $B$ with an involution and a gauge action such that
\begin{itemize}
\item $B_0$ and $B_1$ generate $B$ as a locally convex involutive algebra,
\item all $b$ are gauge smooth and the induced map $B\to C^{\infty}(S^1,B)$ is continuous.
\end{itemize}
\end{defn}

Generalized Weyl algebras $A=\CC[h](\sigma, P)$ are locally convex algebras when given the fine topology. They have an involution defined by $y^*=x$,
$x^*=y$ and $d^*=d$ for all $d\in D$. There is an action of $S^{1}$ defined by $\gamma_t(\omega_n)=t^n\omega_n$ for $\omega_n\in A_n$. With this action,
 generalized Weyl algebras over $\CC[h]$ are smooth generalized crossed products.

\begin{rmk}\label{GWA_not_tame_smooth}
Generalized Weyl algebras $A=\CC[h](\sigma, P)$ are only tame smooth when $P$ is constant (see definition 18 in \cite{MR3054304}). If $P$ is non constant,
 we have $A_1A_{-1}=(P)\subsetneq A_0=\CC[h]$. This implies that $A$ is not tame smooth because tame smooth generalized crossed products $B$ have a frame
 in degree $1$ which implies that $B_1B_{-1}=B_0$.
\end{rmk}

\begin{defn}
Let $B$ be a smooth generalized crossed product. We define $\mcT_B$ to be the closed subalgebra of $\mcT \otimes_{\pi} B$ generated by
$1\otimes B_0, S\otimes B_1$ and $S^{*}\otimes B_{-1}$.
\end{defn}

We tensor the linearly split extension $0\to \mcK\to \mcT\to C^{\infty}(S^{1})\to 0$ with $B$ to obtain
\begin{equation}\label{toeplitz_tensor_B}
0 \to \mcK\otimes_{\pi} B \to \mcT \otimes_{\pi} B \overset{p}\to  C^{\infty}(S^1) \otimes_{\pi} B \to 0.
\end{equation}
which is still a linearly split extension.

\begin{prop}\label{Toeplitzext}
Let $A$ be a generalized Weyl algebra $\CC[h](\sigma, P)$. Then there is a linearly split extension
$$
0 \to \Lambda_A \overset{\iota}\to  \mcT_A \overset{\bar p}\to  A \to 0
$$
where $\Lambda_A$ is the ideal $\bigoplus_{i, j\geq 0}  e_{i, j}\otimes A_{i+1}A_{-(j+1)}$ of $\mcT_A$, $\iota$ is the inclusion of $\Lambda_A$ in $\mcT_A$
and $\bar p$ is the restriction of $p$ to $\mcT_A$.
\end{prop}

\begin{proof}
By Corollary \ref{GWA_countable_basis}, $A$ has a countable basis over $\CC$. Using Lemma \ref{algtensor}, we conclude that the projective tensor products of
 $(\ref{toeplitz_tensor_B})$ are all algebraic. The image of $\mcT_A$ is generated by $1\otimes A_0$, $z\otimes A_1$ and $z^{-1}\otimes A_{-1}$ and it is
 isomorphic to $A$ via $z^{n}\otimes a_n \mapsto a_n$. The kernel of $\pi$ is the intersection of $\mcK\otimes A$ and $\mcT_A$. The elements of $\mcT_A$ are of
 the form $1\otimes a_0+\sum_{k, l\geq 0} S^{k+1}S^{*l+1}\otimes a_{k+1}a_{-(l+1)}$. Now we note that
$$S^{k+1}S^{l+1}=
\begin{cases}
S^{k-l}(1-e_{0, 0}-\dots -e_{l, l}) &,  k\geq l \\
(1-e_{0, 0}-\dots -e_{k, k})S^{*(l-k)} &,  k<l
\end{cases}
$$
Using the vector space decomposition $\mcT\otimes A=\mcK\otimes A\oplus C^{\infty}(S^1)\otimes A$ we note that an element of the kernel is of the form
$$\sum_{k\geq l} S^{k-l}(-e_{0, 0}-\dots -e_{l, l})\otimes a_{k+1}a_{-(l+1)}+\sum_{k<l}(-e_{0, 0}-\dots -e_{k, k})S^{*(l-k)}\otimes a_{k+1}a_{-(l+1)}$$
which is in $\bigoplus_{i, j\geq 0}  e_{i, j}\otimes A_{i+1}A_{-(j+1)}$.

\end{proof}


\subsection{The case where $P$ is a non constant polynomial}\label{P_non_constant}\

We consider the short exact sequence
$$0 \to \Lambda_A \to \mcT_A \to A \to 0$$
where $\mcT_A$ is the Toeplitz algebra of $A$. This sequence yields an exact triangle
$$S A\to \Lambda_A\to \mcT_A\to A$$
 in the triangulated category $(\mathfrak{KK}^{alg}, S)$.

In order to apply Lemma \ref{representation}, we consider generalized Weyl algebras $A=\CC[h](\sigma, P)$ where $P$ is non constant and
\begin{itemize}
\item $q=1$ and $h_0\neq 0$ or
\item $q$ is not a root of unity and $P$ has a root different from $\frac{h_0}{1-q}$.
\end{itemize}
By Propositions~\ref{classical_and_quantum} and~\ref{canonical_GWA}, in order to cover these cases, it suffices to consider the following two cases:
\begin{itemize}
\item  $\sigma(h)=h-1$ and $P$ is a non constant polynomial with $P(0)=0$.
\item  $\sigma(h)=qh$ where $q$ is not a root of unity and $P$ is a non constant polynomial with $P(1)=0$.
\end{itemize}
In both cases, we construct an exact triangle
\begin{equation*}
SA \to A_1A_{-1}\overset{0}\to A_0 \to A
\end{equation*}
in the triangulated category $(\mathfrak{KK}^{alg}, S)$.

\begin{rmk}
We treat the case where $q$ is not a root of unity and $P$ has only $\frac{h_0}{1-q}$ as a root separately in Proposition~\ref{case4}.
\end{rmk}

We start by characterizing the elements of $A_1A_{-1}$.

\begin{lem}
The elements of $\Lambda_A$ can be written uniquely as sums $\sum e_{i, j}\otimes y^{i+1}P_{i,j}(h)x^{j+1}$
\end{lem}

\begin{proof}
Follows from Lemma \ref{GWA-grading}.
\end{proof}

Define $j_{1}:A_1A_{-1}\to \Lambda_A$ by $j_{1}(a)=e_{00}\otimes a$. We embed $\Lambda_A$ in a suitable algebra so that we can construct a Morita equivalence
 to its subalgebra $j_1(A_1A_{-1})=e_{00}\otimes A_1A_{-1}$. Consider the faithful representation $\rho:A\to \mcE$ from Lemma \ref{representation} (where $\mcE=\mcE_1$ if $q=1$ and $\mcE=\mcE_2$ if $q\neq 1$). Tensoring with $1_{\mcT}$ we obtain an injective morphism $1_{\mcT} \otimes \rho:\mcT\otimes A \to \mcT\otimes \mcE$ which restricts to an injective morphism    $\bar{\rho}: \mcT_A\hookrightarrow \mcT \otimes \mcE$.

Now, we show $\Lambda_A\cong_{\mathfrak{KK}^{alg}} A_1A_{-1}$.

\begin{thm}\label{moritaker}
There is a Morita equivalence between $\Lambda_A$ and $j_1(A_1A_{-1})$, therefore there is an invertible element
 $\theta\in kk^{alg}(\Lambda_A, A_1A_{-1})$ which is an inverse of $kk(j_1)$.
\end{thm}

\begin{proof}
We write the proof in the case $\sigma(h)=h-1$ and $P(0)=0$. The case $\sigma(h)=qh$ and $P(1)=0$ can be proven in a similar way since the matrices involved
in the representation satisfy corresponding algebraic relations (Lemma \ref{relations-N-and-G}).

Using the representation from Lemma \ref{representation}, we obtain a faithful representation
$$\bar\rho:\Lambda_A\to\mcT \otimes \mcE_1.$$
The Morita equivalence is given by $\xi_i=\xi'_i=e_{i, 0}\otimes \mcU^i_1$ and $\eta_j=\eta'_j=e_{0,j}\otimes \mcU^i_{-1}$. We check that these sequences
 satisfy the conditions in Definition \ref{MoritaDef} and Proposition \ref{MoritaCond}.

First, we establish that $\xi_i$, $\eta_j$ defines a Morita context between $\Lambda_A$ and $e_{00}\otimes A_1A_{-1}$ according with
Definition~\ref{MoritaDef}. Let $w=\sum e_{i, j}\otimes y^{i+1}P_{i,j}(h)x^{j+1}$ be an element of $\Lambda_A$.

\begin{enumerate}
\item $\eta_j\bar{\rho}(w)\xi_i\in e_{00}\otimes A_1A_{-1}$. We have
$$\eta_j\bar{\rho}(w)\xi_i = e_{00}\otimes \mcU_{-1}^j[P(N)\mcU_1]^{j+1}P_{j, i}(N)\mcU_{-1}^{i+1}\mcU_{1}^i$$
Note that $(P(N)\mcU_1)^{j+1}=\mcU_1^{j+1}R_{j+1}(N)$ where
$$R_{j+1}(N)=P(\sigma(N))\dots P(\sigma^{j+1}(N))=P(\sigma(N))R'_{j+1}(N)$$
and therefore we have
\begin{align*}
\eta_j\bar{\rho}(w)\xi_i &= e_{00}\otimes \mcU_{1}R_{j+1}(N)P_{j, i}(N)\mcU_{-1}\\
&= e_{00}\otimes \mcU_{1}P(\sigma(N))R'_{j+1}(N)P_{j, i}(N)\mcU_{-1}\\
&= e_{00}\otimes P(N)\mcU_{1}R'_{j+1}(N)P_{j, i}(N)\mcU_{-1}\\
&=\bar{\rho}(e_{00}\otimes yR'_{j+1}(h)P_{j, i}(h)x) \in \bar{\rho}(e_{00}\otimes A_1A_{-1})
\end{align*}

\item The terms $\eta_j\bar{\rho}(w)\xi_i$ are rapidly decreasing. This is because the elements of $\Lambda_A$ are finite sums.

\item $(\sum \xi_i\eta_i)\bar{\rho}(w) = \bar\rho(w)$. We have
\begin{align*}
(\sum \xi_i\eta_i)\bar{\rho}(w) &=
\left(\sum e_{i,i}\otimes \mcU_1^i\mcU_{-1}^{i}\right)\left(\sum e_{k, l}\otimes (\mcU_1P(N))^{k+1}P_{k, l}(N)U_{-1}^{l+1}\right) \\
&= \sum{} e_{k, l}\otimes \mcU_1^k\mcU_{-1}^{k}\mcU_1^{k+1}R_{k+1}(N)P_{k, l}(N)\mcU_{-1}^{l+1} \\
&= \sum e_{k, l}\otimes \mcU_1^{k+1}R_{k+1}(N)P_{k, l}(N)\mcU_{-1}^{l+1}\\
&= \bar{\rho}(w)
\end{align*}
\end{enumerate}

Now we check the conditions of  Proposition \ref{MoritaCond}. We show that $\bar{\rho}(w)\xi_k\xi'_l$ and $\eta'_k\eta'_l\bar{\rho}(w)$ are still elements of
$\bar{\rho}(\Lambda_A)$.
\begin{align*}
\bar{\rho}(w)\xi_k\xi'_l &= \left(\sum e_{i, j}\otimes (\mcU_1P(N))^{i+1}P_{i, j}(N)U_{-1}^{j+1}\right)(e_{k, 0}\otimes \mcU_1^k)(e_{l, 0}\otimes \mcU_1^l)\\
 \intertext{which is $0$ unless $l=0$ and in that case we obtain}
\bar{\rho}(w)\xi_k\xi'_l &= \sum e_{i, 0}\otimes (\mcU_1P(N))^{i+1}P_{i, k}(N)U_{-1}\\
&= \bar{\rho}\left(\sum e_{i, 0}\otimes y^{i+1}P_{i, k}(h)x\right)\in \bar{\rho}(\Lambda_A) \\ \intertext{and similarly we compute}
\eta'_k\eta_l\bar{\rho}(w) &= (e_{0, k}\otimes \mcU_{-1}^k)(e_{0, l}\otimes \mcU_{-1}^l)
\left(\sum e_{i, j}\otimes (\mcU_1P(N))^{i+1}P_{i, j}(N)U_{-1}^{j+1}\right)\\ \intertext{which is $0$ unless $k=0$ and in that case we obtain}
\bar{\rho}(w)\xi_k\xi'_l &= \sum e_{0, j}\otimes \mcU_{-1}^l (\mcU_1P(N))^{l+1}P_{l, j}(N)U_{-1}^{j+1}\\
&= \bar{\rho}\left(\sum e_{0, j}\otimes y R'_{l+1}(h)P_{l, j}(h)x^{j+1}\right)\in \bar{\rho}(\Lambda_A).
\end{align*}
The Morita context from $e_{00}\otimes A_1A_{-1}$ to $\Lambda_A$ is defined by  $(\xi'_i, \eta'_j)$. So far we have proved $kk((\xi_i), (\eta_j))\cdot kk((\xi'_i), (\eta'_j))=1_{\Lambda_A}$. Let $z=e_{00}\otimes yP_{0, 0}x\in e_{00}\otimes A_1A_{-1}$. Then $\bar\rho(z)\xi'_l\xi_k=\bar\eta_l\eta'_{k}\rho(z)=0$ unless $l=k=0$ and in this case $\bar\rho(z)\xi'_0\xi_0=\bar\rho(z)\eta_0\eta'_{0}=\bar\rho(z)$. Thus we have $kk((\xi'_i), (\eta'_j))\cdot kk((\xi_i), (\eta_j))=1_{e_{00}\otimes A_1A_{-1}}$.

\end{proof}

Next, we show $\mcT_A\cong_{\mathfrak{KK}^{alg}} A_0$. Define $j_0:A_0\to \mcT_A$ by $j_0(a)=1\otimes a$. We show that this inclusion induces an
invertible element $kk(j_0)\in kk_0^{alg}(\mcT_A,A_0)$.

\begin{lem}
There is a quasihomomorphism $ (id, Ad(S\otimes 1)): \mcT_A \rightrightarrows \mcT\otimes A \triangleright \mcC$, where
$$\mcC=\bigoplus_{i, j\in\NN}  e_{i, j} \otimes A_iA_{-j}.$$
Here $Ad(S\otimes 1)$ is the restriction of $Ad(S\otimes 1):\mcT\otimes A\to \mcT\otimes A$ defined by $x\mapsto (S\otimes 1)x(S^*\otimes 1)$.
\end{lem}
\begin{proof}
We have $A_iA_{-j}A_jA_{-k}\subseteq A_iA_{-k}$ because $A_{-j}A_{j}\subseteq A_0$, therefore $\mcC$ is a subalgebra.  To prove that the pair
$(id, Ad(S\otimes 1))$ defines a quasihomomorphism we check the conditions on the generators. It is clear that
$(1\otimes A_0)\mcC, (S\otimes A_1)\mcC$ and $(S^*\otimes A_{-1})\mcC$ are subsets of $\mcC$. Now we let $a_i\in A_i$ and we check
\begin{eqnarray*}
(id-Ad(S\otimes 1))(1\otimes a_0) &=&  e\otimes a_0 \in \mcC \\
(id-Ad(S\otimes 1))(S\otimes a_1) &=& Se\otimes a_1 \in \mcC \\
(id-Ad(S\otimes 1))(S^*\otimes a_{-1}) &=&  eS^*\otimes a_{-1} \in \mcC.
\end{eqnarray*}
\end{proof}

Define $i_{0}:A_0\to \mcC$ by $i_0(a)=e_{00}\otimes a$.

\begin{prop}\label{Morita2}
There is a Morita equivalence between $\mcC$ and $i_0(A_0)$. Therefore there is an invertible element $\kappa\in kk^{alg}(\mcC, A_0)$.
\end{prop}
\begin{proof}
Using Lemma \ref{representation}, we think of $\mcC$ represented in $\mcT \otimes \mcE$ (where $\mcE=\mcE_1$ if $q=1$ and $\mcE=\mcE_2$ if $q\neq 1$). The Morita equivalence is given by $\xi_i=\xi'_i=e_{i, 0}\otimes \mcU^i_1$ and
$\eta_j=\eta'_j=e_{0,j}\otimes \mcU^i_{-1}$. The proof that these elements determine a Morita equivalence is similar to the proof of
Theorem~\ref{moritaker}.  The Morita context $((\xi_i), (\eta_j))$ determines a morphism $\mcC\to \mcK\otimes_{\pi} i_0(A_0)$ that in
turn determines the element $kk((\xi_i), (\eta_j))\in kk_0^{alg}(\mcC, A_0)$. We define $\kappa=kk((\xi_i), (\eta_j))kk(i_0)^{-1}$ where
$i_0:A_0\to e_{00}\otimes A_0$.

\end{proof}

\begin{prop}\label{LeftInv}
Let $\kappa\in kk(\mcC, A_0)$ as in Proposition \ref{Morita2}, then $$kk(j_0)kk(id, Ad(S\otimes 1))\kappa=1_{A_0}.$$
This implies that $kk(j_0)$ has a right inverse and that $kk(id, Ad(S\otimes 1))$ has a left inverse.
\end{prop}
\begin{proof}
We have
$$(id-Ad(S\otimes 1))(j_0(a_0))=e_{00}\otimes a_0,$$
thus $kk(j_0)kk(id, Ad(S\otimes 1))=kk(i_0)$. By Proposition \ref{Morita2}, $kk(i_0)\kappa=1_{A_0}$.
\end{proof}

To show that $kk(j_0)$ is invertible, we construct a right inverse for $kk(id, Ad(S\otimes 1))$. In order to do this, we construct a diffotopic family of
 quasihomomorphism between $\mcT_A$ and a subalgebra $\bar\mcC$ of $(\mcT\otimes_{\pi} \mcT) \otimes A$ and prove that $\bar\mcC$ is Morita equivalent to
  $\mcT_A$. To construct this diffotopic family we use the diffotopy $\phi_t:\mcT\to \mcT\otimes_{\pi} \mcT$ of Lemma \ref{ToepDiffo}.

Consider the map $\Phi_t=\phi_t\otimes id_A:\mcT\otimes A\to (\mcT\otimes_{\pi} \mcT)\otimes A$ where $\phi_t$ is the diffotopy of lemma \ref{ToepDiffo}. Since $\phi_0(S)=S\otimes 1$, then $\Phi_0(x\otimes a)=x\otimes 1\otimes a$.

\begin{lem}\label{Phi_t_diffeo}
There is a diffotopic family of quasihomomorphisms
$$
(\Phi_t, \Phi_0 \circ Ad(S\otimes 1)): \mcT_A \rightrightarrows (\mcT\otimes_{\pi}\mcT)\otimes A \triangleright \bar\mcC.
$$
Where $\bar{\mcC}$ is the subalgebra
$$\bigoplus_{i, j, p, q \in\NN} e_{i, j}\otimes S^pS^{*q}\otimes A_{i+p}A_{-(j+q)}.$$
\end{lem}

\begin{proof}
We check that $(\Phi_t, \Phi_0 \circ Ad(S\otimes 1))$ define quasihomomorphisms using the generators of $\mcT_A$. First we note that
$\Phi_t(1\otimes A_0)\bar{\mcC}, \Phi_t(S\otimes A_1)\bar{\mcC}$ and $\Phi_t(S^*\otimes A_{-1})\bar{\mcC}$ are subsets of $\bar{\mcC}$. Finally, we compute
\begin{eqnarray*}
(\Phi_t-\Phi_0 \circ Ad(S\otimes 1))(1\otimes a_0) &=&  e\otimes 1\otimes  a_0 \in \bar{\mcC} \\
(\Phi_t-\Phi_0 \circ Ad(S\otimes 1))(S\otimes a_1)&=& f(t)(e\otimes S\otimes a_1)+ g(t)(Se\otimes 1\otimes a_1) \in \bar{\mcC}  \\
(\Phi_t-\Phi_0 \circ Ad(S\otimes 1))(S^*\otimes a_{-1})&=&  \bar{f}(t)(e\otimes S^*\otimes a_{-1})+ \bar{g}(t)(eS^*\otimes 1\otimes a_{-1})  \in \bar{\mcC}.
\end{eqnarray*}
\end{proof}

Define $\eta:\mcT_A\to \bar\mcC$ as the restriction of the injective morphism $\mcT\otimes A\to (\mcT\otimes_{\pi} \mcT)\otimes A$ given by
$\eta(x\otimes a)= e\otimes x\otimes a$.
\begin{prop}\label{eta_is_iso}
There is a Morita equivalence between $\bar\mcC$ and $\eta(\mcT_A)$. Therefore $kk(\eta)\in kk_0^{alg}(\mcT_A, \bar\mcC)$ is invertible.
\end{prop}

\begin{proof}
Using Lemma \ref{representation}, we have an injective morphism $\bar\mcC\to(\mcT\otimes_{\pi} \mcT) \otimes \mcE$ (where $\mcE=\mcE_1$ if $q=1$ and $\mcE=\mcE_2$ if $q\neq 1$)The Morita equivalence is given by $\xi_i=\xi'_i=e_{i, 0}\otimes 1\otimes \mcU^i_1$ and $\eta_j=\eta'_j=e_{0,j}\otimes 1\otimes \mcU^j_{-1}$. The proof is
similar to the proof of Theorem \ref{moritaker}.
\end{proof}

\begin{thm}\label{toep_grado0}
$kk(j_0)\in kk^{alg}(A_0, \mcT_A)$ is invertible.
\end{thm}

\begin{proof}
By \ref{LeftInv}, we know that $kk(j_0)$ has a right inverse and $kk(id, Ad(S\otimes 1))$ has a left inverse. Now, we prove that $k(id, Ad(S\otimes 1))$
has a right inverse which completes the proof.

Since $\phi_0(S)=S\otimes 1$, then if $a_i\in A_{i}$ and $a_{-j}\in A_{-j}$, we have
$$\Phi_0(e_{i, j}\otimes a_ia_{-j})=e_{i, j}\otimes 1\otimes a_ia_{-j}\in \bar{\mcC}$$
and therefore $\Phi_0(\mcC)\subseteq \bar{\mcC}$, thus by item $(4)$ of Proposition \ref{quasi-homo} we have
$$kk(id, Ad(S\otimes 1)) kk(\Phi_0|_{\mcC}) = kk(\Phi_0, \Phi_0\circ Ad(S\otimes 1)).$$
By item $(5)$ of Proposition \ref{quasi-homo}, we obtain
$$ kk(\Phi_0, \Phi_0\circ Ad(S\otimes 1))= kk(\Phi_1, \Phi_0\circ Ad(S\otimes 1)).$$
We have $\phi_1(S)=S^{2}S^{*}\otimes 1+e\otimes S$ and therefore $\Phi_1 - \Phi_0\circ Ad(S\otimes 1) = \eta$. By item $(2)$ of Proposition \ref{quasi-homo},  $kk(\Phi_1, \Phi_0\circ Ad(S\otimes 1))=kk(\eta)$ and by Lemma \ref{eta_is_iso}, $kk(\eta)$ is invertible.
\end{proof}

With the isomorphisms in $\mathfrak{KK}^{alg}$ fromTheorems \ref{moritaker} and  \ref{toep_grado0}, we construct the desired exact triangle.

\begin{thm}\label{exact_triangle}
For a generalized Weyl algebra $A=\CC[h](\sigma, P(h))$ with $P$ a non constant polynomial and
\begin{itemize}
\item $q=1$ and $h_0\neq 0$ or
\item $q$ is not a root of unity and $P$ has a root different from $\frac{h_0}{1-q}$
\end{itemize}
there is an exact triangle
$$
S A\to A_1A_{-1}\overset{0}\to A_0 \to A .
$$
\end{thm}

\begin{proof}
The linearly split extension
\begin{equation}\label{extension_toeplitz_algebra}
0 \to \Lambda_A \overset{\iota}\to  \mcT_A \overset{\bar p}\to A \to 0
\end{equation}
yields an exact triangle
$$
S A \overset{kk(E)}\to  \Lambda_A \overset{kk(\iota)}\to \mcT_A \overset{kk(\bar p)}\to A,
$$
where $kk(E)\in kk_1^{alg}(A, \Lambda_A)=kk_0^{alg}(S A, \Lambda_A)$ is the element defined by the extension (\ref{extension_toeplitz_algebra}).

By Theorem \ref{moritaker}, the inclusion $j_1:A_1A_{-1}\to \Lambda_A$ defined by $j_1(x)=e_{00}\otimes x$ induces an invertible element
$kk(j_1)\in kk_0^{alg}(A_1A_{-1}, \Lambda_{A})$. By Theorem ~\ref{toep_grado0}, the inclusion $j_0:A_0\to \mcT_A$ defined by $j_0(a)=1\otimes a$
induces an invertible element $kk(j_0)\in kk_0^{alg}(A_0, \mcT_A)$. We define $\phi$ by the commutative diagram in $\mathfrak{KK}^{alg}$
$$
\xymatrix{
\Lambda_A\ar[r]^{kk(\iota)} & \mcT_A\ar[d]^{kk(j_0)^{-1}}\\
A_1A_{-1}\ar[u]^{kk(j_1)}\ar[r]_{\phi} & A_0
}
$$
and claim that
\begin{equation}\label{diferencia}
  \phi= kk(i)-kk(\sigma).
\end{equation}
For this we use Proposition \ref{LeftInv} to obtain
$$kk(j_0)^{-1}=kk(id, Ad(S\otimes 1))\kappa$$
and therefore
$$
kk(j_1)kk(\iota)kk(j_0)^{-1}=kk(j_1)kk(\iota)kk(id, Ad(S\otimes 1))\kappa.
$$

Let $x=R(h)\in A_1A_{-1}\subseteq \CC[h]$. The product $kk(j_1)kk(\iota)kk(id, Ad(S\otimes 1))$
corresponds to the quasihomomorphism $(\phi, \psi): A_1A_{-1} \rightrightarrows \mcT\otimes A\triangleright \mcC$, where $\phi(x)=e_{00}\otimes x$ and
 $\psi(x)=e_{11}\otimes x$. Since $\phi$ and $\psi$ are orthogonal, we obtain $kk(\phi, \psi)=kk(\phi)-kk(\psi)$. Now we multiply this difference by $\kappa$
 which is given by the Morita equivalence of Proposition \ref{Morita2}. Thus we have that $kk(\phi)\kappa$ and $kk(\psi)\kappa$ are determined by maps
 $A_1A_{-1}\to \mcC\to \mcK\otimes A_0$ that send $x\mapsto e_{00}\otimes x$ and
 $x\mapsto e_{00}\otimes \rho^{-1}(\mcU_{-1}R(G)\mcU_{1})=e_{00}\otimes R(\sigma(h))$ (here we use the representation $\rho$ of Lemma \ref{representation}).
 Thus we conclude $\phi=kk(i)-kk(\sigma)$, proving~\eqref{diferencia}.

Now we prove that $\phi=0$. Both $i_1$ and $\sigma$ factor through a contractible subalgebra of $\CC[h]$. This is because we have $i_1(A_1A_{-1})=P(h)\CC[h]$
and $\sigma(A_1A_{-1})=P(\sigma(h))\CC[h]$ and the polynomials $P(h)$ and $P(\sigma(h))$ have some linear factors $L(h)$ and $L(\sigma(h))$. Thus the
morphisms $i_1$ and $\sigma$ factor through the subalgebras $L(h)\CC[h]$ and $L(\sigma(h))\CC[h]$ which are contractible. Therefore we have
$kk(i_1)=kk(\sigma)=0$.
\end{proof}

\begin{lem}\label{lemma_triangulated}
Let $(\mathfrak{T}, \Sigma)$ be a triangulated category. If there is an exact triangle
$$
\Sigma X\to Y \overset{0}\to  Z \to X
$$
then $X\cong Z\oplus \Sigma^{-1}Y$.
\end{lem}
\begin{proof}
See Corollary 1.2.7 in \cite{MR1812507}.
\end{proof}

Now we compute the isomorphism class of $A_1A_{-1}$ in $\mathfrak{KK}^{alg}$.

\begin{prop}\label{invariants_A1A-1}
Let $A=\CC[h](\sigma, P)$ where $P$ is a nonconstant polynomial with $r$ different roots, then
$$A_1A_{-1}\cong_{\mathfrak{KK}^{alg}} S\CC^{r-1}.$$
\end{prop}

\begin{proof}
Let $P(h)=c(h-h_1)^{n_1}\dots (h-h_r)^{n_r}$. Without loss of generality we can assume $c=1$. Since $A_1A_{-1}=(P(h))$ we have a linearly split extension
\begin{equation}\label{sequence_computations_1}
0\to A_1A_{-1}\to \CC[h]\overset{\pi}\to  \CC[h]/(P(h))\to 0.
\end{equation}
By the Chinese Reminder Theorem, there is an isomorphism
$$\phi: \CC[h]/(P(h))\to \prod_{i=1}^{r}  \CC[h]/(h-h_i)^{n_i}.$$

We have the following commutative diagram
$$
\xymatrix{
0\ar[r] & (h-h_i)^{n_i}\CC[h]\ar[r]\ar[d] & \CC[h]\ar[r]^{q_{i}}\ar[d]_{=} & \CC[h]/(h-h_i)^{n_i}\ar[r]\ar[d]_{\mu_i} & 0\\
0\ar[r] & (h-h_i)\CC[h]\ar[r] & \CC[h]\ar[r]_{ev_{h_i}} & \CC\ar[r] & 0
}
$$
Since $(h-h_i)^{n_i}\CC[h]$ and $(h-h_i)\CC[h]$ are contractible, $kk(q_i)$ and $kk(ev_{h_i})$ are invertible, therefore
$kk(\mu_i)\in kk_0^{alg}( \CC[h]/(h-h_i)^{n_i}, \CC)$ is invertible. By the additivity of $\mathfrak{KK}^{alg}$, the homomorphism
$\mu:\prod_{i=1}^{r}  \CC[h]/(h-h_i)^{n_i}\to \CC^{r}$
given by $\mu_i$ in the $i$-th component induces an invertible element $kk(\mu)$. Note that $\mu\circ \pi:\CC[h]\to \CC^{r}$ is given by $ev_{h_i}$ in the
$i$-th component.

Since all evaluation maps $ev_{h_i}$ induce the same $kk^{alg}$-isomorphism ${kk(ev_0)}$ in $kk^{alg}(\CC[h],\CC)$, we have the commutative diagram in
$\mathfrak{KK}^{alg}$
$$
\xymatrix
{
 \CC[h]\ar[r]^{kk(\pi)\ \ \ \ }\ar[d]_{kk(ev_0)} & \CC[h]/P(h)\ar[d]^{kk(\mu)}\\
\CC \ar[r]_{kk(\triangle)}& \CC^{r}
}
$$
where $\triangle:\CC\to\CC^{r}$ is the diagonal morphism $\triangle(1)=(1, \dots, 1)$. Replacing $\CC[h]$ by $\CC$ and $\CC[h]/(P(h))$ by $\CC^r$ in
the exact triangle corresponding to~\eqref{sequence_computations_1},
 we obtain an exact triangle in $\mathfrak{KK}^{alg}$
\begin{equation}\label{triangle_1}
S \CC^{r}\to A_1A_{-1}\to \CC \overset{kk(\triangle)}\to  \CC^{r}.
\end{equation}
The linearly split extension $ 0 \to  \CC \overset{\triangle}\to \CC^{r} \to \CC^{r-1} \to 0$ yields an exact triangle
$$
S \CC^{r-1} \to \CC \overset{kk(\triangle)}\to \CC^{r} \to \CC^{r-1}.
$$
Permuting this triangle we obtain the exact triangle
\begin{equation}\label{triangle_2}
S \CC^{r} \to S \CC^{r-1} \to  \CC \overset{kk(\triangle)}\to \CC^{r}.
\end{equation}
Since both triangles (\ref{triangle_1}) and (\ref{triangle_2}) complete the morphism $kk(\triangle):\CC\to\CC^{r}$, by the axiom TR3 of triangulated categories
we have $A_1A_{-1} \cong_{\mathfrak{KK}^{alg}} S\CC^{r-1}$.
\end{proof}

\begin{thm}\label{main_result}
Let $A=\CC[h](\sigma, P(h))$ be generalized Weyl with $\sigma(h)=qh+h_0$ and $P$ a non constant polynomial such that
\begin{itemize}
\item $q=1$ and $h_0\neq 0$ or
\item $q$ is not a root of unity and $P$ has a root different from $\frac{h_0}{1-q}$,
\end{itemize}
then $A\cong_{\mathfrak{KK}^{alg}} \CC^{r}$.
\end{thm}

\begin{proof}
The result follows from Theorem \ref{exact_triangle}, Lemma \ref{lemma_triangulated} and Proposition \ref{invariants_A1A-1}.
\end{proof}

\begin{cor}\label{result_in_Lp}
Let $A$ be as in Theorem \ref{main_result}. Then $A\cong \CC^{r}$ in $\mathfrak{KK}^{\mathcal{L}_p}$ and so
$$
\qquad\qquad\qquad\qquad\qquad kk_0^{\mathcal{L}_p}(\CC, A)=\ZZ^{r} \quad \text{and} \quad kk_1^{\mathcal{L}_p}(\CC, A)=0.\qquad\qquad\qquad\qquad\qquad\qed
$$
\end{cor}

Corollary \ref{result_in_Lp} implies $K_0(A\otimes_{\pi}\mathcal{L}_p)=\ZZ^{r}$. This is compatible with Theorem 4.5 of \cite{MR1247356},
 which computes $K_0(A)=\ZZ^{r}$ for $A=\CC[h](\sigma, P)$ when $\sigma(h)=h-1$ and $P$ has $r$ simple roots.

\begin{exs} We apply Theorem \ref{main_result} in the following cases.
\begin{enumerate}
\item The quantum Weyl algebra $A_q$ with $q\neq 1$ not a root of unity, is isomorphic to $\CC$ in $\mathfrak{KK}^{alg}$.
\item In the case of the primitive factors $B_{\lambda}$ of $U(\mathfrak{sl}_2)$, we have $P(h)=-h(h+1)-\lambda/4$. If $\lambda= 1$, then
$B_{\lambda}\cong \CC$ in $\mathfrak{KK}^{alg}$. If $\lambda\neq 1$, then $B_{\lambda}\cong \CC^{2}$ in $\mathfrak{KK}^{alg}$. This implies
 $kk_0^{\mathcal{L}_p}(\CC, B_{\lambda})=\ZZ\oplus \ZZ$ and $kk_1^{\mathcal{L}_p}(\CC, B_{\lambda})=0$.
\item The quantum weighted projective line  $\mathcal{O}(\mathbb{WP}_q(k, l))$ is  isomorphic to $\CC[h](\sigma, P)$ with $\sigma(h)=q^{2l}h$ and
$$P(h)=h^{k}\prod_{i=0}^{l-1}(1-q^{-2i}h).$$
In the case $q\neq 1$ is not a root of unity, we have $\mathcal{O}(\mathbb{WP}_q(k, l))\cong \CC^{l+1}$ in
$\mathfrak{KK}^{alg}$. This implies $kk_0^{\mathcal{L}_p}(\CC,\mathcal{O}(\mathbb{WP}_q(k, l)) )=\ZZ^{l+1}$ and
$kk_1^{\mathcal{L}_p}(\CC,\mathcal{O}(\mathbb{WP}_q(k, l)) )=0$. (Compare with Corollary 5.3 of \cite{MR2989456}.)
\end{enumerate}
\end{exs}

In the case where $q$ is not a root of unity and $P$ has only $\frac{h_0}{1-q}$ as a root we have the following result.

\begin{prop}\label{case4}
The generalized Weyl algebra $A=\CC[h](\sigma, P(h))$, with $\sigma(h)=qh+h_0$ such that $q\neq 1$ and $P$ has only $\frac{h_0}{1-q}$ as a root, is isomorphic
to $\CC$ in $\mathfrak{KK}^{alg}$.
\end{prop}

\begin{proof}
By Proposition \ref{classical_and_quantum}, $A$ is isomorphic to $\CC[h](\sigma_1, P_1)$ with $\sigma_1(h)=qh$ and $P_1(h)=ch^{n}$  with $c\in\CC^{*}$ and
$n\geq 1$. The algebra $\CC[h](\sigma_1, P_1)$ is $\NN$ graded with $\deg h=2$, $\deg x=n$ and $\deg y=n$. To prove this we check that the defining relations
$$xh=qhx,\  yh=q^{-1}hy,\ yx=ch^{n} \text{ and } xy=cq^{n}h^{n}$$
are compatible with the grading.

The result follows from Lemma \ref{N-graded}, since the degree $0$ subalgebra of $A$ is equal to $\CC$.
\end{proof}

\begin{ex}
The quantum plane $\CC[h](\sigma, h)$ with $\sigma(h)=qh$ is isomorphic to $\CC$ in $\mathfrak{KK}^{alg}$.
\end{ex}


\subsection{The case where P is a constant polynomial}\label{case2}\

If $P$ is a nonzero constant polynomial, then $A=\CC[h](\sigma, P)$ is a tame smooth generalized crossed product and we can apply the results
from~\cite{MR3054304}.

\begin{prop}\label{invariants-exceptional-case-1}
Let $A=\CC[h](\sigma, P)$ where $P\neq 0$ is a constant polynomial, then $A\cong S\CC\times \CC$ in $\mathfrak{KK}^{alg}$. This implies
$A\cong S\CC\times \CC$ in $\mathfrak{KK}^{\mathcal{L}_p}$ and therefore we have $kk_0^{\mathcal{L}_p}(\CC, A)=\ZZ$ and $kk_1^{\mathcal{L}_p}(\CC, A)=\ZZ$.
\end{prop}

\begin{proof}
In this case $A$ is a tame smooth generalized crossed product with
frame $\xi_i=y^{i}$ and $\bar\xi_i=x^{i}$ for $i\in \NN$. This frame satisfies the conditions of Definition 18 in \cite{MR3054304}, therefore we have a
linearly split extension
\begin{equation*}
0 \to \Lambda_A \overset{\iota}\to   \mcT_A \overset{\bar p}\to  A \to 0,
\end{equation*}
that yields an exact triangle
$$
S A \overset{kk(E)}\to   \Lambda_A \overset{kk(\iota)}\to \mcT_A \overset{kk(\bar p)}\to A
$$

By Theorem 27 of \cite{MR3054304}, $j_1:\CC[h]\to \Lambda_A$, defined by $j_1(x)= e_{00}\otimes x$ induces an invertible element $kk(j_1)$ and by Theorem 33
of~\cite{MR3054304}, $j_0:\CC[h]\to \mcT_A$ defined by $j_0(x)= 1\otimes x$ induces an invertible element $kk(j_0)$. We have a commutative diagram in
$\mathfrak{KK}^{alg}$
$$
\xymatrix{
\Lambda_A\ar[r]_{kk(\iota)} & \mcT_A\\
\CC[h]\ar[u]{kk(j_{1})}\ar[r]_{\alpha} & \CC[h]\ar[u]_{kk(j_0)}.
}
$$

We prove that $\alpha=1_{\CC[h]}-kk(\sigma)$ and that $1_{\CC[h]}=kk(\sigma)$, thus concluding that $\alpha=0$. By Theorem 33 of \cite{MR3054304}, we have
$kk(j_0)^{-1}=kk(id, Ad(S\otimes 1))kk(\iota_{1})^{-1}$. The product $kk(j_1)kk(\iota)kk(1, Ad(S\otimes 1))$ corresponds to a quasihomomorphism
$$
(\phi, \psi): \CC[h] \rightrightarrows \mcT\otimes A\triangleright \mcC,
$$
where $\phi(Q)=e_{00}\otimes Q$ and $\psi(Q)=e_{11}\otimes Q$ for all $Q\in \CC[h]$. Since $\phi$ and $\psi$ are orthogonal $kk(\phi, \psi)=kk(\phi)-kk(\psi)$.
We now compose $kk(\phi)$ and $kk(\psi)$ with $kk(j_1)^{-1}$ . Theorem 27 of \cite{MR3054304} characterizes $kk({j_1})^{-1}$ as given by a Morita equivalence
defined by
$$\Xi_i=S^{i}\otimes y^{i}\quad \text{ and }\quad\overline\Xi_i=S^{*i}\otimes x^{i}.$$
therefore $kk(\phi)kk(j_1)^{-1}$ is defined by the morphism $Q\mapsto Q$ and  $kk(\psi)kk(j_1)^{-1}$ is defined by $Q\mapsto xQy=\sigma(Q)$. This implies that
$\alpha=1_{\CC[h]}-kk(\sigma)$.

The commutative diagram
$$
\xymatrix{
\CC[h]\ar[r]_{\sigma} \ar[d]_{ev_0}& \CC[h]\ar[d]_{ev_0}\\
\CC\ar[r]_{id} &\CC,
}
$$
implies that $kk(\sigma)=1_{\CC[h]}$ and thus $\alpha=0$.

This implies the existence of an exact triangle in $\mathfrak{KK}^{alg}$
$$
S A \to  \CC \overset{0}\to \CC \to A.
$$
Using Lemma \ref{lemma_triangulated}, we obtain $A\cong S\CC\oplus \CC$ in $\mathfrak{KK}^{alg}$.
\end{proof}

In the case where $P=0$ we have the following result.

\begin{prop}\label{case1}
The generalized Weyl algebra $A=\CC[h](\sigma, P(h))$ with $P=0$ is isomorphic to $\CC$ in $\mathfrak{KK}^{alg}$.
\end{prop}

\begin{proof}
The relations
$$xh=\sigma(h)x,\  yh=\sigma^{-1}(h)y,\ yx=0 \text{ and } xy=0$$
are compatible with the grading determined by $\deg h=0$, $\deg x=1$ and $\deg y=1$, therefore the algebra $A$ is $\NN$-graded. The result follows from Lemma \ref{N-graded} and the fact that the degree $0$ subalgebra of $A$ is equal to $\CC[h]$.
\end{proof}

\end{document}